\documentclass[11pt]{preprint}
\usepackage[full]{textcomp}
\usepackage[osf]{newtxtext} 
\usepackage[cal=boondoxo]{mathalfa}
\usepackage{colortbl}
\usepackage{shuffle}

\usepackage{comment}

\usepackage{amssymb}
\usepackage{lmodern}
\usepackage{mathtools}

\usepackage{hyperref}
\usepackage{breakurl}
\usepackage{aligned-overset}
\usepackage{mhenvs}
\usepackage{mhequ} 
\newcommand{\be}{\begin{equation*}}
	\newcommand{\ee}{\end{equation*}}
\usepackage{mhsymb}
\usepackage{booktabs}
\usepackage{tikz}
\usepackage{tcolorbox}
\usepackage{mathrsfs}
\usepackage[utf8]{inputenc}
\usepackage{longtable}
\usepackage{wrapfig}
\usepackage{subcaption}
\usepackage{mathrsfs}
\usepackage{epsfig}
\usepackage{microtype}
\usepackage{comment}
\usepackage{wasysym}
\usepackage{centernot}
\usepackage{enumitem}
\usepackage{bm}
\usepackage{stackrel}
\usepackage{graphicx}
\usepackage{tikz-cd}
\usepackage{quiver}

\makeatletter
\newcommand{\globalcolor}[1]{%
	\color{#1}\global\let\default@color\current@color
}
\makeatother

\usetikzlibrary{calc}
\usetikzlibrary{decorations}
\usetikzlibrary{positioning}
\usetikzlibrary{shapes}
\usetikzlibrary{shapes.misc}

\tikzset{cross/.style={cross out, draw=black, fill=none, minimum size=2*(#1-\pgflinewidth), inner sep=0pt, outer sep=0pt}, cross/.default={2pt}}

\definecolor{blush}{rgb}{0.87, 0.36, 0.51}
\definecolor{brightcerulean}{rgb}{0.11, 0.67, 0.84}
\definecolor{greenryb}{rgb}{0.4, 0.69, 0.2}

\newif\ifdark
\darkfalse

\ifdark
\definecolor{darkred}{rgb}{0.9,0.2,0.2}
\definecolor{darkblue}{rgb}{0.7,0.3,1}
\definecolor{darkgreen}{rgb}{0.1,0.9,0.1}
\definecolor{franck}{rgb}{0,0.8,1}
\definecolor{pagebackground}{rgb}{.15,.21,.18}
\definecolor{pageforeground}{rgb}{.84,.84,.85}
\pagecolor{pagebackground}
\AtBeginDocument{\globalcolor{pageforeground}}
\definecolor{symbols}{rgb}{0,0.7,1}
\colorlet{connection}{red!80!black}
\colorlet{boxcolor}{blue!50}

\else

\definecolor{darkred}{rgb}{0.7,0.1,0.1}
\definecolor{darkblue}{rgb}{0.4,0.1,0.8}
\definecolor{darkgreen}{rgb}{0.1,0.7,0.1}
\definecolor{franck}{rgb}{0,0,1}
\definecolor{pagebackground}{rgb}{1,1,1}
\definecolor{pageforeground}{rgb}{0,0,0}
\colorlet{symbols}{blue!90!black}
\colorlet{connection}{red!30!black}
\colorlet{boxcolor}{blue!50!black}

\fi

\def\slash{\leavevmode\unskip\kern0.18em/\penalty\exhyphenpenalty\kern0.18em}
\def\dash{\leavevmode\unskip\kern0.18em--\penalty\exhyphenpenalty\kern0.18em}

\DeclareMathAlphabet{\mathbbm}{U}{bbm}{m}{n}

\DeclareFontFamily{U}{BOONDOX-calo}{\skewchar\font=45 }
\DeclareFontShape{U}{BOONDOX-calo}{m}{n}{
	<-> s*[1.05] BOONDOX-r-calo}{}
\DeclareFontShape{U}{BOONDOX-calo}{b}{n}{
	<-> s*[1.05] BOONDOX-b-calo}{}
\DeclareMathAlphabet{\mcb}{U}{BOONDOX-calo}{m}{n}
\SetMathAlphabet{\mcb}{bold}{U}{BOONDOX-calo}{b}{n}
\setlist{noitemsep,topsep=4pt,leftmargin=1.5em}

\DeclareMathAlphabet{\mathbbm}{U}{bbm}{m}{n}

\DeclareMathAlphabet{\mcb}{U}{BOONDOX-calo}{m}{n}
\SetMathAlphabet{\mcb}{bold}{U}{BOONDOX-calo}{b}{n}
\DeclareFontFamily{U}{mathx}{\hyphenchar\font45}
\DeclareFontShape{U}{mathx}{m}{n}{
	<5> <6> <7> <8> <9> <10>
	<10.95> <12> <14.4> <17.28> <20.74> <24.88>
	mathx10
}{}
\DeclareSymbolFont{mathx}{U}{mathx}{m}{n}
\DeclareMathSymbol{\bigtimes}{1}{mathx}{"91}

\providecommand{\figures}{false}
{ \ifthenelse{\equal{\figures}{false}} {#1}{\[ {\rm Figure \ missing !} \]} }{}

\newcommand{\trigl}{\vartriangleleft}

\usepackage{mathtools}

\tikzstyle{tinydots}=[dash pattern=on \pgflinewidth off \pgflinewidth]
\tikzstyle{superdense}=[dash pattern=on 4pt off 1pt]

\newcommand{\beq}{\begin{equation}}
	\newcommand{\eeq}{\end{equation}}

\usepackage{empheq}

\newcommand\Item[1][]{%
	\ifx\relax#1\relax  \item \else \item[#1] \fi
	\abovedisplayskip=0pt\abovedisplayshortskip=0pt~\vspace*{-\baselineskip}}

\def\${|\!|\!|}

\newcounter{definition}

\newtheorem{ex}[definition]{Example}

\usepackage{thmtools}
\theoremstyle{definition}

\newenvironment{DIFnomarkup}{}{} % see man latexdiff

\theorembodyfont{\rmfamily}

\newfont{\indic}{bbmss12}

\def\Nabla_#1{\nabla_{\!#1}}

\makeatletter
\pgfdeclareshape{crosscircle}
{
	\inheritsavedanchors[from=circle] % this is nearly a circle
	\inheritanchorborder[from=circle]
	\inheritanchor[from=circle]{north}
	\inheritanchor[from=circle]{north west}
	\inheritanchor[from=circle]{north east}
	\inheritanchor[from=circle]{center}
	\inheritanchor[from=circle]{west}
	\inheritanchor[from=circle]{east}
	\inheritanchor[from=circle]{mid}
	\inheritanchor[from=circle]{mid west}
	\inheritanchor[from=circle]{mid east}
	\inheritanchor[from=circle]{base}
	\inheritanchor[from=circle]{base west}
	\inheritanchor[from=circle]{base east}
	\inheritanchor[from=circle]{south}
	\inheritanchor[from=circle]{south west}
	\inheritanchor[from=circle]{south east}
	\inheritbackgroundpath[from=circle]
	\foregroundpath{
		\centerpoint%
		\pgf@xc=\pgf@x%
		\pgf@yc=\pgf@y%
		\pgfutil@tempdima=\radius%
		\pgfmathsetlength{\pgf@xb}{\pgfkeysvalueof{/pgf/outer xsep}}%  
		\pgfmathsetlength{\pgf@yb}{\pgfkeysvalueof{/pgf/outer ysep}}%  
		\ifdim\pgf@xb<\pgf@yb%
		\advance\pgfutil@tempdima by-\pgf@yb%
		\else%
		\advance\pgfutil@tempdima by-\pgf@xb%
		\fi%
		\pgfpathmoveto{\pgfpointadd{\pgfqpoint{\pgf@xc}{\pgf@yc}}{\pgfqpoint{-0.707107\pgfutil@tempdima}{-0.707107\pgfutil@tempdima}}}
		\pgfpathlineto{\pgfpointadd{\pgfqpoint{\pgf@xc}{\pgf@yc}}{\pgfqpoint{0.707107\pgfutil@tempdima}{0.707107\pgfutil@tempdima}}}
		\pgfpathmoveto{\pgfpointadd{\pgfqpoint{\pgf@xc}{\pgf@yc}}{\pgfqpoint{-0.707107\pgfutil@tempdima}{0.707107\pgfutil@tempdima}}}
		\pgfpathlineto{\pgfpointadd{\pgfqpoint{\pgf@xc}{\pgf@yc}}{\pgfqpoint{0.707107\pgfutil@tempdima}{-0.707107\pgfutil@tempdima}}}
	}
}
\makeatother

\def\symbol#1{\textcolor{symbols}{#1}}

\def\decorate#1#2{
	\ifnum#2>0
	\foreach \count in {1,...,#2}{
		let
		\p1 = (sourcenode.center),
		\p2 = (sourcenode.east),
		\n1 = {\x2-\x1},
		\n2 = {1mm},
		\n3 = {(1.3+0.6*(\count-1))*\n1},
		\n4 = {0.7*\n1}
		in 
		node[rectangle,fill=symbols,rotate=30,inner sep=0pt,minimum width=0.2*\n2,minimum height=\n2] at ($(sourcenode.center) + (\n3,\n4)$) {}
	}
	\fi
	\ifnum#1>0
	\foreach \count in {1,...,#1}{
		let
		\p1 = (sourcenode.center),
		\p2 = (sourcenode.east),
		\n1 = {\x2-\x1},
		\n2 = {1mm},
		\n3 = {(1.3+0.6*(\count-1))*\n1},
		\n4 = {0.7*\n1}
		in 
		node[rectangle,fill=symbols,rotate=-30,inner sep=0pt,minimum width=0.2*\n2,minimum height=\n2] at ($(sourcenode.center) + (-\n3,\n4)$) {}
	}
	\fi
}

\tikzset{
	dectriangle/.style 2 args={
		triangle,
		alias=sourcenode,
		append after command={\decorate{#1}{#2}}
	},
	dectriangle/.default={0}{0},
}

\tikzset{
	cross/.style={path picture={ 
			\draw[symbols]
			(path picture bounding box.south east) -- (path picture bounding box.north west) (path picture bounding box.south west) -- (path picture bounding box.north east);
	}},
	root/.style={circle,fill=green!50!black,inner sep=0pt, minimum size=1.2mm},
	dot/.style={circle,fill=pageforeground,inner sep=0pt, minimum size=1mm},
	dotred/.style={circle,fill=pageforeground!50!pagebackground,inner sep=0pt, minimum size=2mm},
	var/.style={circle,fill=pageforeground!10!pagebackground,draw=pageforeground,inner sep=0pt, minimum size=3mm},
	kernel/.style={semithick,shorten >=2pt,shorten <=2pt},
	kernels/.style={snake=zigzag,shorten >=2pt,shorten <=2pt,segment amplitude=1pt,segment length=4pt,line before snake=2pt,line after snake=5pt,},
	rho/.style={densely dashed,semithick,shorten >=2pt,shorten <=2pt},
	testfcn/.style={dotted,semithick,shorten >=2pt,shorten <=2pt},
	renorm/.style={shape=circle,fill=pagebackground,inner sep=1pt},
	labl/.style={shape=rectangle,fill=pagebackground,inner sep=1pt},
	xic/.style={very thin,circle,draw=symbols,fill=symbols,inner sep=0pt,minimum size=1.2mm},
	g/.style={very thin,rectangle,draw=symbols,fill=symbols!10!pagebackground,inner sep=0pt,minimum width=2.5mm,minimum height=1.2mm},
	xi/.style={very thin,circle,draw=symbols,fill=symbols!10!pagebackground,inner sep=0pt,minimum size=1.2mm},
	xies/.style={very thin,rectangle,fill=green!50!black!25,draw=symbols,inner sep=0pt,minimum size=1.1mm},
	xiesf/.style={very thin,rectangle,fill=green!50!black,draw=symbols,inner sep=0pt,minimum size=1.1mm},
	xix/.style={very thin,crosscircle,fill=symbols!10!pagebackground,draw=symbols,inner sep=0pt,minimum size=1.2mm},
	X/.style={very thin,cross,rectangle,fill=pagebackground,draw=symbols,inner sep=0pt,minimum size=1.2mm},
	xib/.style={thin,circle,fill=symbols!10!pagebackground,draw=symbols,inner sep=0pt,minimum size=1.6mm},
	xie/.style={thin,circle,fill=green!50!black,draw=symbols,inner sep=0pt,minimum size=1.6mm},
	xid/.style={thin,circle,fill=symbols,draw=symbols,inner sep=0pt,minimum size=1.6mm},
	xibx/.style={thin,crosscircle,fill=symbols!10!pagebackground,draw=symbols,inner sep=0pt,minimum size=1.6mm},
	kernels2/.style={very thick,draw=connection,segment length=12pt},
	keps/.style={thin,draw=symbols,->},
	kepspr/.style={thick,draw=connection,->},
	krho/.style={thin,draw=symbols,superdense,->},
	krhopr/.style={thick,draw=connection,superdense},
	triangle/.style = { regular polygon, regular polygon sides=3},
	not/.style={thin,circle,draw=connection,fill=connection,inner sep=0pt,minimum size=0.5mm},
	diff/.style = {very thin,draw=symbols,triangle,fill=red!50!black,inner sep=0pt,minimum size=1.6mm},
	diff1/.style = {very thin,dectriangle={1}{0},fill=red!50!black,draw=symbols,inner sep=0pt,minimum size=1.6mm},
	diff2/.style = {very thin,dectriangle={1}{1},fill=red!50!black,draw=symbols,inner sep=0pt,minimum size=1.6mm},
	diffmini/.style = {very thin,rectangle,fill=black,draw=black,inner sep=0pt,minimum size=0.75mm},
	kernelsmod/.style={very thick,draw=connection,segment length=12pt},
	rec/.style = {very thin,rectangle,fill=black,draw=black,inner sep=0pt,minimum size=2mm},
	cerc/.style={very thin,circle,draw=black,fill=symbols,inner sep=0pt,minimum size=2mm},
	stars/.style={very thin,star,star points=6,star point ratio=0.5, draw=black,fill=red,inner sep=0pt,minimum size=0.7mm},
	>=stealth,
}
\tikzset{
	root/.style={circle,fill=black!50,inner sep=0pt, minimum size=3mm},
	circ/.style={circle,fill=white,draw=black,very thin,inner sep=.5pt, minimum size=1.2mm},
	round1/.style={fill=white,outer sep = 0,inner sep=2pt,rounded corners=1mm,draw,text=black,thin,minimum size=1.2mm},
	circ1/.style={circle,fill=red!10,draw=red,very thin,inner sep=.5pt, minimum size=1.2mm},
	rect/.style={fill=white,outer sep = 0,inner sep=2pt,rectangle,draw,text=black,thin,minimum size=1.2mm},
	rect1/.style={fill=white,outer sep = 0,inner sep=2pt,rectangle,draw,text=black,thin,minimum size=1.2mm},
	round2/.style={fill=red!10,outer sep = 0,inner sep=2pt,rounded corners=1mm,draw,text=black,thin,minimum size=1.2mm},
	round3/.style={fill=blue!10,outer sep = 0,inner sep=2pt,rounded corners=1mm,draw,text=black,thin,minimum size=1.2mm}, 
	rect2/.style={fill=black!10,outer sep = 0,inner sep=2pt,rectangle,draw,text=black,thin,minimum size=1.2mm},
	dot/.style={circle,fill=black,inner sep=0pt, minimum size=1.2mm},
	dotred/.style={circle,fill=black!50,inner sep=0pt, minimum size=2mm},
	var/.style={circle,fill=black!10,draw=black,inner sep=0pt, minimum size=3mm},
	kernel/.style={semithick,shorten >=2pt,shorten <=2pt},
	diag/.style={thin,shorten >=4pt,shorten <=4pt},
	kernel1/.style={thick},
	kernels/.style={snake=zigzag,shorten >=2pt,shorten <=2pt,segment amplitude=1pt,segment length=4pt,line before snake=2pt,line after snake=5pt,},
	kernels1/.style={snake=zigzag,segment amplitude=0.5pt,segment length=2pt},
	rho1/.style={densely dotted,semithick},
	rho/.style={densely dashed,semithick,shorten >=2pt,shorten <=2pt},
	testfcn/.style={dotted,semithick,shorten >=2pt,shorten <=2pt},
	visible/.style={draw, circle, fill, inner sep=0.25ex},
	renorm/.style={shape=circle,fill=white,inner sep=1pt},
	labl/.style={shape=rectangle,fill=white,inner sep=1pt},
	xic/.style={very thin,circle,fill=symbols,draw=black,inner sep=0pt,minimum size=1.2mm},
	xi/.style={very thin,circle,fill=blue!10,draw=black,inner sep=0pt,minimum size=1.2mm},
	xib/.style={very thin,circle,fill=blue!10,draw=black,inner sep=0pt,minimum size=1.6mm},
	xie/.style={very thin,circle,fill=green!50!black,draw=black,inner sep=0pt,minimum size=1mm},
	xid/.style={very thin,circle,fill=symbols,draw=black,inner sep=0pt,minimum size=1.6mm},
	edgetype/.style={very thin,circle,draw=black,inner sep=0pt,minimum size=5mm},
	nodetype/.style={very thick,circle,draw=black,inner sep=0pt,minimum size=5mm},
	kernels2/.style={very thick,draw=connection,segment length=12pt},
	clean/.style={thin,circle,fill=black,inner sep=0pt,minimum size=1mm},	not/.style={thin,circle,fill=symbols,draw=connection,fill=connection,inner sep=0pt,minimum size=0.8mm},
	>=stealth,
}

\makeatletter
\def\DeclareSymbol#1#2#3{%
	\expandafter\gdef\csname MH@symb@#1\endcsname{\tikzsetnextfilename{symbol#1}%
		\tikz[baseline=#2,scale=0.15,draw=symbols,line join=round]{#3}}%
	\expandafter\gdef\csname MH@symb@#1s\endcsname{\scalebox{0.75}{\tikzsetnextfilename{symbol#1}%
			\tikz[baseline=#2,scale=0.15,draw=symbols,line join=round]{#3}}}%
	\expandafter\gdef\csname MH@symb@#1ss\endcsname{\scalebox{0.65}{\tikzsetnextfilename{symbol#1}%
			\tikz[baseline=#2,scale=0.15,draw=symbols,line join=round]{#3}}}%
}
\def\<#1>{\ifthenelse{\boolean{mmode}}{\mathchoice{\csname MH@symb@#1\endcsname}{\csname MH@symb@#1\endcsname}{\csname MH@symb@#1s\endcsname}{\csname MH@symb@#1ss\endcsname}}{\csname MH@symb@#1\endcsname}}
\makeatother

\DeclareSymbol{Xi22}{0.5}{\draw (0,0) node[xi] {} -- (-1,1) node[not] {} -- (0,2) node[xi] {};} % 1 not used in the text
\DeclareSymbol{Xi2}{-2}{\draw (-1,-0.25) node[xi] {} -- (0,1) node[xi] {};} % 2
\DeclareSymbol{Xi2b}{-2}{\draw (-1,-0.25) node[xic] {} -- (0,1) node[xic] {};} % 2
\DeclareSymbol{Xi2g}{-2}{\draw (-1,-0.25) node[xies] {} -- (0,1) node[xi] {};} % 2
\DeclareSymbol{Xi2g2}{-2}{\draw (-1,-0.25) node[xi] {} -- (0,1) node[xies] {};} % 2
\DeclareSymbol{cXi2}{-2}{\draw (0,-0.25) node[xi] {} -- (-1,1) node[xic] {};}%3
\DeclareSymbol{Xi3}{0}{\draw (0,0) node[xi] {} -- (-1,1) node[xi] {} -- (0,2) node[xi] {};}%4
\DeclareSymbol{XiIIXi}{0}{\draw (0,0) node[xi] {} -- (-1,1); \draw[kernels2] (-1,1) node[not] {} -- (0,2) node[xi] {};}%4

\DeclareSymbol{Xi4}{2}{\draw (0,0) node[xi] {} -- (-1,1) node[xi] {} -- (0,2) node[xi] {} -- (-1,3) node[xi] {};}%5
\DeclareSymbol{Xi4_1}{2}{\draw (0,0) node[xic] {} -- (-1,1) node[xic] {} -- (0,2) node[xi] {} -- (-1,3) node[xi] {};}%6
\DeclareSymbol{Xi4_2}{2}{\draw (0,0) node[xic] {} -- (-1,1) node[xi] {} -- (0,2) node[xi] {} -- (-1,3) node[xic] {};}%7
\DeclareSymbol{Xi2X}{-2}{\draw (0,-0.25) node[xi] {} -- (-1,1) node[xix] {};}%8
\DeclareSymbol{XXi2}{-2}{\draw (0,-0.25) node[xix] {} -- (-1,1) node[xi] {};}%9
\DeclareSymbol{IIXi}{0}{\draw (0,-0.25) node[not] {} -- (-1,1) node[xi] {} -- (0,2) node[xi] {};}%10
\DeclareSymbol{IXi^2}{-1}{\draw (-1,1) node[xi] {} -- (0,0) node[not] {} -- (1,1) node[xi] {};}%11
\DeclareSymbol{IIXi^2}{-4}{\draw (0,-1.5) node[not] {} -- (0,0);
	\draw[kernels2] (-1,1) node[xi] {} -- (0,0) node[not] {} -- (1,1) node[xi] {};}%12
\DeclareSymbol{XiX}{-2.8}{\node[xibx] {};}%13
\DeclareSymbol{tauX}{-2.8}{ \node[X] {};}%14
\DeclareSymbol{Xi}{-2.8}{\node[xib] {};}%15

\DeclareSymbol{IXiX}{-1}{\draw (0,-0.25) node[not] {} -- (-1,1) node[xix] {};}%18
\DeclareSymbol{IXi3}{2}{\draw (0,-0.25) node[not] {} -- (-1,1) node[xi] {} -- (0,2) node[xi] {} -- (-1,3) node[xi] {};}%20
\DeclareSymbol{IXi}{-2}{\draw (0,-0.25) node[not] {} -- (-1,1) node[xi] {};}%21
\DeclareSymbol{XiI}{-2}{\draw (0,-0.25) node[xi] {} -- (-1,1) node[not] {};}%22

\DeclareSymbol{Xi4b}{0}{\draw(0,1.5) node[xi] {} -- (0,0); \draw (-1,1) node[xi] {} -- (0,0) node[xi] {} -- (1,1) node[xi] {};}%23
\DeclareSymbol{Xi4b'}{0}{\draw(0,1.5) node[xi] {} -- (0,-0.2); \draw (-1,1) node[xi] {} -- (0,-0.2) node[not] {} -- (1,1) node[xi] {};}%24 not used
\DeclareSymbol{Xi4c}{0}{\draw (0,1) -- (0.8,2.2) node[xi] {};\draw (0,-0.25) node[xi] {} -- (0,1) node[xi] {} -- (-0.8,2.2) node[xi] {};}%25
\DeclareSymbol{Xi4d}{-4.5}{\draw (0,-1.5) node[not] {} -- (0,0); \draw (-1,1) node[xi] {} -- (0,0) node[xi] {} -- (1,1) node[xi] {};}%26 not used
\DeclareSymbol{Xi4e}{0}{\draw (0,2) node[xi] {} -- (-1,1) node[xi] {} -- (0,0) node[xi] {} -- (1,1) node[xi] {};}%27
\DeclareSymbol{Xi4e'}{0}{\draw (0,2) node[xi] {} -- (-1,1) node[xi] {} -- (0,-0.2) node[not] {} -- (1,1) node[xi] {};}%28 not used

\DeclareSymbol{Xitwo}%29
{0}{\draw[kernels2] (0,0) node[not] {} -- (-1,1) node[not] {}
	-- (-2,2) node[not]{} -- (-3,3) node[xi]  {};
	\draw[kernels2] (0,0) -- (1,1) node[xi] {};
	\draw[kernels2] (-1,1) -- (0,2) node[xi] {};
	\draw[kernels2] (-2,2) -- (-1,3) node[xi] {};}

\DeclareSymbol{IXitwo}%30
{0}{\draw (-.7,1.2) node[xi] {} -- (0,-0.2) -- (.7,1.2) node[xi] {};}
\DeclareSymbol{I1Xitwo}%30
{0}{\draw[kernels2] (0,0) node[not] {} -- (-1,1) node[xi] {};
	\draw[kernels2] (0,0) -- (1,1) node[xi] {};}

\DeclareSymbol{I1Xitwobis}%30
{0}{\draw[kernels2] (0,0) node[not] {} -- (-1,1) node[xies] {};
	\draw[kernels2] (0,0) -- (1,1) node[xies] {};}

\DeclareSymbol{I1Xitwog}%30
{0}{\draw[kernels2] (0,0) node[not] {} -- (-1,1) node[xies] {};
	\draw[kernels2] (0,0) -- (1,1) node[xi] {};}

\DeclareSymbol{cI1Xitwo}%31
{0}{\draw[kernels2] (0,0) node[not] {} -- (-1,1) node[xic] {};
	\draw[kernels2] (0,0) -- (1,1) node[xi] {};}

\DeclareSymbol{I1IXi3}{0}{\draw (0,0) node[xi] {} -- (-1,1) ; %32
	\draw[kernels2] (-1,1) node[not] {} -- (0,2) node[xi] {};
	\draw[kernels2] (-1,1) node[not] {} -- (-2,2) node[xi] {};}

\DeclareSymbol{I1Xi3c}{-1}{\draw[kernels2](0,1.5) node[xi] {} -- (0,0) node[not] {}; \draw (-1,1) node[xi] {} -- (0,0) ; \draw[kernels2] (0,0) -- (1,1) node[xi] {};}%33

\DeclareSymbol{I1Xi3cbis}{-1}{\draw[kernels2](0,1.5) node[xies] {} -- (0,0) node[not] {}; \draw (-1,1) node[xies] {} -- (0,0) ; \draw[kernels2] (0,0) -- (1,1) node[xies] {};}%33

\DeclareSymbol{I1IXi3b}{0}{\draw[kernels2] (0,0) node[not] {} -- (-1,1) ; \draw[kernels2] (0,0)   -- (1,1) node[xi] {} ;
	\draw (-1,1) node[xi] {} -- (0,2) node[xi] {};
}%34

\DeclareSymbol{I1IXi3c}{0}{\draw[kernels2] (0,0) node[not] {} -- (-1,1) ; \draw[kernels2] (0,0)   -- (1,1) node[xi] {} ;
	\draw[kernels2] (-1,1) node[not] {} -- (0,2) node[xi] {};
	\draw[kernels2] (-1,1) node[not] {} -- (-2,2) node[xi] {};}%35

\DeclareSymbol{I1IXi3cbis}{0}{\draw[kernels2] (0,0) node[not] {} -- (-1,1) ; \draw[kernels2] (0,0)   -- (1,1) node[xies] {} ;
	\draw[kernels2] (-1,1) node[not] {} -- (0,2) node[xies] {};
	\draw[kernels2] (-1,1) node[not] {} -- (-2,2) node[xies] {};}%35

\DeclareSymbol{I1Xi}{0}{\draw[kernels2] (0,0) node[not] {} -- (-1,1)  node[xi] {} ;}%36

\DeclareSymbol{I1Xi4a}{2}{\draw[kernels2] (0,0) node[not] {} -- (-1,1) ; \draw[kernels2] (0,0) node[not] {} -- (1,1) node[xi] {} ;%37
	\draw (-1,1) node[xi] {} -- (0,2) node[xi] {} -- (-1,3) node[xi] {};}
\DeclareSymbol{cI1Xi4a}{2}{\draw[kernels2] (0,0) node[not] {} -- (-1,1) ; \draw[kernels2] (0,0) node[not] {} -- (1,1) node[xic] {} ;%39
	\draw (-1,1) node[xic] {} -- (0,2) node[xi] {} -- (-1,3) node[xi] {};}
\DeclareSymbol{I1Xi4b}{2}{\draw (0,0) node[xi] {} -- (-1,1) node[xi] {} -- (0,2) ; \draw[kernels2] (0,2) node[not] {} -- (-1,3) node[xi] {};\draw[kernels2] (0,2)  -- (1,3) node[xi] {};
}%41

\DeclareSymbol{cI1Xi4b}{2}{\draw (0,0) node[xic] {} -- (-1,1) node[xic] {} -- (0,2) ; \draw[kernels2] (0,2) node[not] {} -- (-1,3) node[xi] {};\draw[kernels2] (0,2)  -- (1,3) node[xi] {};
}%42

\DeclareSymbol{I1Xi4c}{2}{\draw (0,0) node[xi] {} -- (-1,1) node[not] {}; \draw[kernels2] (-1,1) -- (0,2) ; 
	\draw[kernels2] (-1,1) -- (-2,2) node[xi] {} ;
	\draw (0,2) node[xi] {} -- (-1,3) node[xi] {};}%43

\DeclareSymbol{cI1Xi4c}{2}{\draw (0,0) node[xic] {} -- (-1,1) node[not] {}; \draw[kernels2] (-1,1) -- (0,2) ; 
	\draw[kernels2] (-1,1) -- (-2,2) node[xic] {} ;
	\draw (0,2) node[xi] {} -- (-1,3) node[xi] {};}%44

\DeclareSymbol{I1Xi4ab}{2}{\draw[kernels2] (0,0) node[not] {} -- (-1,1) ; \draw[kernels2] (0,0) node[not] {} -- (1,1) node[xi] {};\draw (-1,1) node[xi] {} -- (0,2) ; \draw[kernels2] (0,2) node[not] {} -- (-1,3) node[xi] {};\draw[kernels2] (0,2)  -- (1,3) node[xi] {}; }%45

\DeclareSymbol{cI1Xi4ab}{2}{\draw[kernels2] (0,0) node[not] {} -- (-1,1) ; \draw[kernels2] (0,0) node[not] {} -- (1,1) node[xic] {};\draw (-1,1) node[xic] {} -- (0,2) ; \draw[kernels2] (0,2) node[not] {} -- (-1,3) node[xi] {};\draw[kernels2] (0,2)  -- (1,3) node[xi] {}; }%46

\DeclareSymbol{I1Xi4bc}{2}{\draw (0,0) node[xi] {} -- (-1,1) node[not] {}; \draw[kernels2] (-1,1) -- (0,2) ; %47
	\draw[kernels2] (-1,1) -- (-2,2) node[xi] {} ; \draw[kernels2] (0,2) node[not] {} -- (-1,3) node[xi] {};\draw[kernels2] (0,2)  -- (1,3) node[xi] {};
}

\DeclareSymbol{cI1Xi4bc}{2}{\draw (0,0) node[xic] {} -- (-1,1) node[not] {}; \draw[kernels2] (-1,1) -- (0,2) ; %48
	\draw[kernels2] (-1,1) -- (-2,2) node[xic] {} ; \draw[kernels2] (0,2) node[not] {} -- (-1,3) node[xi] {};\draw[kernels2] (0,2)  -- (1,3) node[xi] {};
}

\DeclareSymbol{I1Xi4abcc1}{2}{\draw[kernels2] (0,0) node[not] {} -- (-1,1) node[not] {}%50
	-- (-2,2) node[not]{} -- (-3,3) node[xic]  {};
	\draw[kernels2] (0,0) -- (1,1) node[xic] {};
	\draw[kernels2] (-1,1) -- (0,2) node[xi] {};
	\draw[kernels2] (-2,2) -- (-1,3) node[xi] {};
}

\DeclareSymbol{I1Xi4abcc1b}{2}{\draw[kernels2] (0,0) node[not] {} -- (-1,1) node[not] {}%50
	-- (-2,2) node[not]{} -- (-3,3) node[xi]  {};
	\draw[kernels2] (0,0) -- (1,1) node[xic] {};
	\draw[kernels2] (-1,1) -- (0,2) node[xic] {};
	\draw[kernels2] (-2,2) -- (-1,3) node[xi] {};
}

\DeclareSymbol{I1Xi4abcc2}{2}{\draw[kernels2] (0,0) node[not] {} -- (-1,1) node[not] {}%51
	-- (-2,2) node[not]{} -- (-3,3) node[xic]  {};
	\draw[kernels2] (0,0) -- (1,1) node[xi] {};
	\draw[kernels2] (-1,1) -- (0,2) node[xi] {};
	\draw[kernels2] (-2,2) -- (-1,3) node[xic] {};
}

\DeclareSymbol{I1Xi4ac}{2}{\draw[kernels2] (0,0) node[not] {} -- (-1,1) ; \draw[kernels2] (0,0) node[not] {} -- (1,1) node[xi] {}; 
	\draw[kernels2] (-1,1) node[not] {} -- (0,2) ; %52
	\draw[kernels2] (-1,1) -- (-2,2) node[xi] {} ;
	\draw (0,2) node[xi] {} -- (-1,3) node[xi] {};}

\DeclareSymbol{cI1Xi4ac}{2}{\draw[kernels2] (0,0) node[not] {} -- (-1,1) ; \draw[kernels2] (0,0) node[not] {} -- (1,1) node[xic] {}; 
	\draw[kernels2] (-1,1) node[not] {} -- (0,2) ; %53
	\draw[kernels2] (-1,1) -- (-2,2) node[xic] {} ;
	\draw (0,2) node[xi] {} -- (-1,3) node[xi] {};}

\DeclareSymbol{I1Xi4acc1}{2}{\draw[kernels2] (0,0) node[not] {} -- (-1,1) ; \draw[kernels2] (0,0) node[not] {} -- (1,1) node[xic] {}; 
	\draw[kernels2] (-1,1) node[not] {} -- (0,2) ; %54
	\draw[kernels2] (-1,1) -- (-2,2) node[xi] {} ;
	\draw (0,2) node[xic] {} -- (-1,3) node[xi] {};}

\DeclareSymbol{I1Xi4acc2}{2}{\draw[kernels2] (0,0) node[not] {} -- (-1,1) ; \draw[kernels2] (0,0) node[not] {} -- (1,1) node[xic] {}; 
	\draw[kernels2] (-1,1) node[not] {} -- (0,2) ; %55
	\draw[kernels2] (-1,1) -- (-2,2) node[xi] {} ;
	\draw (0,2) node[xi] {} -- (-1,3) node[xic] {};}

\DeclareSymbol{2I1Xi4}{2}{\draw[kernels2] (0,0) node[not] {} -- (-1,1) node[not] {};
	\draw[kernels2] (0,0) -- (1,1) node[not] {};%56
	\draw[kernels2] (-1,1) -- (-1.5,2.5) node[xi] {};
	\draw[kernels2] (-1,1) -- (-0.5,2.5) node[xi] {};
	\draw[kernels2] (1,1) -- (0.5,2.5) node[xi] {};
	\draw[kernels2] (1,1) -- (1.5,2.5) node[xi] {};
}

\DeclareSymbol{2I1Xi4dis}{2}{\draw[kernels2] (0,0) node[not] {} -- (-1,1) node[not] {};
	\draw[kernels2] (0,0) -- (1,1) node[not] {};%56
	\draw[kernels2] (-1,1) -- (-1.5,2.5) node[xies] {};
	\draw[kernels2] (-1,1) -- (-0.5,2.5) node[xies] {};
	\draw[kernels2] (1,1) -- (0.5,2.5) node[xies] {};
	\draw[kernels2] (1,1) -- (1.5,2.5) node[xies] {};
}

\DeclareSymbol{2I1Xi4c1}{2}{\draw[kernels2] (0,0) node[not] {} -- (-1,1) node[not] {};%57
	\draw[kernels2] (0,0) -- (1,1) node[not] {};
	\draw[kernels2] (-1,1) -- (-1.5,2.5) node[xic] {};
	\draw[kernels2] (-1,1) -- (-0.5,2.5) node[xi] {};
	\draw[kernels2] (1,1) -- (0.5,2.5) node[xic] {};
	\draw[kernels2] (1,1) -- (1.5,2.5) node[xi] {};
}

\DeclareSymbol{2I1Xi4c2}{2}{\draw[kernels2] (0,0) node[not] {} -- (-1,1) node[not] {};%58
	\draw[kernels2] (0,0) -- (1,1) node[not] {};
	\draw[kernels2] (-1,1) -- (-1.5,2.5) node[xic] {};
	\draw[kernels2] (-1,1) -- (-0.5,2.5) node[xic] {};
	\draw[kernels2] (1,1) -- (0.5,2.5) node[xi] {};
	\draw[kernels2] (1,1) -- (1.5,2.5) node[xi] {};
}

\DeclareSymbol{2I1Xi4b}{2}{\draw[kernels2] (0,0) node[not] {} -- (-1,1) ;%59
	\draw[kernels2] (0,0) -- (1,1);
	\draw (-1,1) node[xi] {} -- (-1,2.5) node[xi] {};
	\draw (1,1)  node[xi] {} -- (1,2.5) node[xi] {};
}

\DeclareSymbol{2I1Xi4bb}{2}{\draw[kernels2] (0,0) node[not] {} -- (-1,1) ;%59
	\draw[kernels2] (0,0) -- (1,1);
	\draw (-1,1) node[xi] {} -- (-1,2.5) node[xiesf] {};
	\draw (1,1)  node[xi] {} -- (1,2.5) node[xic] {};
}

\DeclareSymbol{2I1Xi4c}{2}{\draw[kernels2] (0,0) node[not] {} -- (-1,1);%60
	\draw[kernels2] (0,0) -- (1,1) node[not] {};
	\draw (-1,1)  node[xi] {} -- (-1,2.5) node[xi] {};
	\draw[kernels2] (1,1) -- (0.4,2.5) node[xi] {};
	\draw[kernels2] (1,1) -- (1.6,2.5) node[xi] {};
}

\DeclareSymbol{2I1Xi4cc1}{2}{\draw[kernels2] (0,0) node[not] {} -- (-1,1);%61
	\draw[kernels2] (0,0) -- (1,1) node[not] {};
	\draw (-1,1)  node[xic] {} -- (-1,2.5) node[xi] {};
	\draw[kernels2] (1,1) -- (0.4,2.5) node[xic] {};
	\draw[kernels2] (1,1) -- (1.6,2.5) node[xi] {};
}

\DeclareSymbol{2I1Xi4cc2}{2}{\draw[kernels2] (0,0) node[not] {} -- (-1,1);%62
	\draw[kernels2] (0,0) -- (1,1) node[not] {};
	\draw (-1,1)  node[xic] {} -- (-1,2.5) node[xic] {};
	\draw[kernels2] (1,1) -- (0.4,2.5) node[xi] {};
	\draw[kernels2] (1,1) -- (1.6,2.5) node[xi] {};
}

\DeclareSymbol{Xi4ba}{0}{\draw(-0.5,1.5) node[xi] {} -- (0,0); \draw (-1.5,1) node[xi] {} -- (0,0) node[not] {}; \draw[kernels2] (0,0) -- (1.5,1) node[xi] {};
	\draw[kernels2] (0,0) -- (0.5,1.5) node[xi] {} ;}%63

\DeclareSymbol{Xi4badis}{0}{\draw(-0.5,1.5) node[xies] {} -- (0,0); \draw (-1.5,1) node[xies] {} -- (0,0) node[not] {}; \draw[kernels2] (0,0) -- (1.5,1) node[xies] {};
	\draw[kernels2] (0,0) -- (0.5,1.5) node[xies] {} ;}%63

\DeclareSymbol{Xi4ba1}{0}{\draw(-0.5,1.5) node[xi] {} -- (0,0); \draw (-1.5,1) node[xi] {} -- (0,0) node[not] {}; \draw[kernels2] (0,0) -- (1.5,1) node[xic] {};
	\draw[kernels2] (0,0) -- (0.5,1.5) node[xic] {} ;}%64

\DeclareSymbol{Xi4ba1b}{0}{\draw(-0.5,1.5) node[xic] {} -- (0,0); \draw (-1.5,1) node[xic] {} -- (0,0) node[not] {}; \draw[kernels2] (0,0) -- (1.5,1) node[xi] {};
	\draw[kernels2] (0,0) -- (0.5,1.5) node[xi] {} ;}%64

\DeclareSymbol{Xi4ba1bdiff}{0}{\draw(-0.5,1.5) node[xic] {} -- (0,0); \draw (-1.5,1) node[xic] {} -- (0,0) node[not] {}; \draw (0,0) -- (1.5,1) node[xi] {};
	\draw (0,0) -- (0.5,1.5) node[xi] {};
	\draw(0,0) node[diff] {};}%64

\DeclareSymbol{Xi4ba1bb}{0}{\draw(-0.5,1.5) node[xic] {} -- (0,0); \draw (-1.5,1) node[xiesf] {} -- (0,0) node[not] {}; \draw[kernels2] (0,0) -- (1.5,1) node[xi] {};
	\draw[kernels2] (0,0) -- (0.5,1.5) node[xi] {} ;}%64

\DeclareSymbol{Xi4ba2}{0}{\draw(-0.5,1.5) node[xi] {} -- (0,0); \draw (-1.5,1) node[xic] {} -- (0,0) node[not] {}; \draw[kernels2] (0,0) -- (1.5,1) node[xi] {};
	\draw[kernels2] (0,0) -- (0.5,1.5) node[xic] {} ;}%65

\DeclareSymbol{Xi4ba2b}{0}{\draw(-0.5,1.5) node[xi] {} -- (0,0); \draw (-1.5,1) node[xic] {} -- (0,0) node[not] {}; \draw[kernels2] (0,0) -- (1.5,1) node[xi] {};
	\draw[kernels2] (0,0) -- (0.5,1.5) node[xiesf] {} ;}%65

\DeclareSymbol{Xi4ca}{0}{\draw (0,1) -- (-1,2.2) node[xi] {};\draw (0,-0.25) node[xi] {} -- (0,1) ; \draw[kernels2] (0,1) node[not] {} -- (1,2.2) node[xi] {};%67
	\draw[kernels2] (0,1) {} -- (0,2.7) node[xi] {};
}

\DeclareSymbol{Xi4cb}{0}{\draw (-1,1) -- (-2,2) node[xi] {};\draw[kernels2] (0,0)  -- (-1,1) node[xi] {} ; \draw[kernels2] (0,0) node[not] {} -- (1,1) node[xi] {} ; 
	\draw (-1,1) node[xi] {} -- (0,2) node[xi] {};}

\DeclareSymbol{Xi4cbb}{0}{\draw (-1,1) -- (-2,2) node[xiesf] {};\draw[kernels2] (0,0)  -- (-1,1) node[xi] {} ; \draw[kernels2] (0,0) node[not] {} -- (1,1) node[xi] {} ; 
	\draw (-1,1) node[xi] {} -- (0,2) node[xic] {};}

\DeclareSymbol{Xi4cbc1}{0}{\draw (-1,1) -- (-2,2) node[xic] {};\draw[kernels2] (0,0)  -- (-1,1) node[xic] {} ; \draw[kernels2] (0,0) node[not] {} -- (1,1) node[xi] {} ; 
	\draw (-1,1) node[xic] {} -- (0,2) node[xi] {};}

\DeclareSymbol{Xi4cbc2}{0}{\draw (-1,1) -- (-2,2) node[xi] {};\draw[kernels2] (0,0)  -- (-1,1) node[xi] {} ; \draw[kernels2] (0,0) node[not] {} -- (1,1) node[xic] {} ; 
	\draw (-1,1) node[xic] {} -- (0,2) node[xi] {};}

\DeclareSymbol{Xi4cab}{0}{\draw (-1,1) -- (-2,2) node[xi] {};\draw[kernels2] (0,0)  -- (-1,1); \draw[kernels2] (0,0) node[not] {} -- (1,1) node[xi] {} ; 
	\draw[kernels2] (-1,1)  {} -- (0,2) node[xi] {};
	\draw[kernels2] (-1,1) node[not] {} -- (-1,2.5) node[xi] {};
}%69

\DeclareSymbol{Xi4cabdis}{0}{\draw (-1,1) -- (-2,2) node[xies] {};\draw[kernels2] (0,0)  -- (-1,1); \draw[kernels2] (0,0) node[not] {} -- (1,1) node[xies] {} ; 
	\draw[kernels2] (-1,1)  {} -- (0,2) node[xies] {};
	\draw[kernels2] (-1,1) node[not] {} -- (-1,2.5) node[xies] {};
}%69

\DeclareSymbol{Xi4cabc1}{0}{\draw (-1,1) -- (-2,2) node[xi] {};\draw[kernels2] (0,0)  -- (-1,1); \draw[kernels2] (0,0) node[not] {} -- (1,1) node[xic] {} ; 
	\draw[kernels2] (-1,1)  {} -- (0,2) node[xic] {};
	\draw[kernels2] (-1,1) node[not] {} -- (-1,2.5) node[xi] {};
}%69

\DeclareSymbol{Xi4cabc2}{0}{\draw (-1,1) -- (-2,2) node[xic] {};\draw[kernels2] (0,0)  -- (-1,1); \draw[kernels2] (0,0) node[not] {} -- (1,1) node[xic] {} ; 
	\draw[kernels2] (-1,1)  {} -- (0,2) node[xi] {};
	\draw[kernels2] (-1,1) node[not] {} -- (-1,2.5) node[xi] {};
}%69

\DeclareSymbol{Xi4ea}{1.5}{\draw (-1,2.5) node[xi] {} -- (-1,1) node[xi] {} -- (0,0); %71
	\draw[kernels2] (0,0)  -- (1,1) node[xi] {};
	\draw[kernels2] (0,0) node[not] {} -- (0,1.5) node[xi] {}; }

\DeclareSymbol{Xi4eac1}{1.5}{\draw (-1,2.5) node[xic] {} -- (-1,1) node[xi] {} -- (0,0); %72
	\draw[kernels2] (0,0)  -- (1,1) node[xic] {};
	\draw[kernels2] (0,0) node[not] {} -- (0,1.5) node[xi] {}; }

\DeclareSymbol{Xi4eac1b}{1.5}{\draw (-1,2.5) node[xic] {} -- (-1,1) node[xi] {} -- (0,0); %72
	\draw[kernels2] (0,0)  -- (1,1) node[xiesf] {};
	\draw[kernels2] (0,0) node[not] {} -- (0,1.5) node[xi] {}; }

\DeclareSymbol{Xi4eac2}{1.5}{\draw (-1,2.5) node[xic] {} -- (-1,1) node[xic] {} -- (0,0); %73
	\draw[kernels2] (0,0)  -- (1,1) node[xi] {};
	\draw[kernels2] (0,0) node[not] {} -- (0,1.5) node[xi] {}; }

\DeclareSymbol{Xi4eact1}{1.5}{\draw (-1,2.5) node[xic] {} -- (-1,1) node[xi] {} -- (0,0); %74
	\draw (0,0)  -- (1,1) node[xic] {};
	\draw[rho] (0,0) node[not] {} -- (0,1.5) node[xi] {}; }

\DeclareSymbol{Xi4eact2}{1.5}{\draw[rho] (-1,2.5) node[xic] {} -- (-1,1) node[xi] {} -- (0,0); %75
	\draw (0,0)  -- (1,1) node[xic] {};
	\draw (0,0) node[not] {} -- (0,1.5) node[xi] {}; }

\DeclareSymbol{Xi4eabis}{1.5}{\draw (-1,2.5) node[xi] {} -- (-1,1) ; \draw[kernels2] (-1,1) node[xi] {} -- (0,0); 
	\draw (0,0)  -- (1,1) node[xi] {};%76
	\draw[kernels2] (0,0) node[not] {} -- (0,1.5) node[xi] {}; }

\DeclareSymbol{Xi4eabisc1}{1.5}{\draw (-1,2.5) node[xic] {} -- (-1,1) ; \draw[kernels2] (-1,1) node[xi] {} -- (0,0); 
	\draw (0,0)  -- (1,1) node[xi] {};%77
	\draw[kernels2] (0,0) node[not] {} -- (0,1.5) node[xic] {}; }

\DeclareSymbol{Xi4eabisc1b}{1.5}{\draw (-1,2.5) node[xic] {} -- (-1,1) ; \draw[kernels2] (-1,1) node[xi] {} -- (0,0); 
	\draw (0,0)  -- (1,1) node[xi] {};%77
	\draw[kernels2] (0,0) node[not] {} -- (0,1.5) node[xiesf] {}; }

\DeclareSymbol{Xi4eabisc1bis}{1.5}{\draw (-1,2.5) node[xi] {} -- (-1,1) ; \draw[kernels2] (-1,1) node[xi] {} -- (0,0); 
	\draw (0,0)  -- (1,1) node[xi] {};%78
	\draw[kernels2] (0,0) node[not] {} -- (0,1.5) node[xi] {};
	\draw (-2,1) node[] {\tiny{$i$}};
	\draw (-2,2.5) node[] {\tiny{$\ell$}};
	\draw (2,1) node[] {\tiny{$k$}};
	\draw (0,2.5) node[] {\tiny{$j$}};
}

\DeclareSymbol{Xi4eabisc1tris}{1.5}{\draw (-1,2.5) node[xi] {} -- (-1,1) ; \draw[kernels2] (-1,1) node[xi] {} -- (0,0); 
	\draw (0,0)  -- (1,1) node[xi] {};%78
	\draw[kernels2] (0,0) node[not] {} -- (0,1.5) node[xi] {};
	\draw (-2,1) node[] {\tiny{i}};
	\draw (-2,2.5) node[] {\tiny{j}};
	\draw (2,1) node[] {\tiny{j}};
	\draw (0,2.5) node[] {\tiny{i}};
}

\DeclareSymbol{Xi4eabisc1quater}{1.5}{\draw (-1,2.5) node[xic] {} -- (-1,1) ; \draw[kernels2] (-1,1) node[xi] {} -- (0,0); 
	\draw (0,0)  -- (1,1) node[xic] {};%78
	\draw[kernels2] (0,0) node[not] {} -- (0,1.5) node[xi] {};
}

\DeclareSymbol{Xi4eabisc2}{1.5}{\draw (-1,2.5) node[xic] {} -- (-1,1) ; \draw[kernels2] (-1,1) node[xi] {} -- (0,0); %79
	\draw (0,0)  -- (1,1) node[xic] {};
	\draw[kernels2] (0,0) node[not] {} -- (0,1.5) node[xi] {}; }

\DeclareSymbol{Xi4eabisc2l}{1.5}{\draw (-1,2.5) node[xiesf] {} -- (-1,1) ; \draw[kernels2] (-1,1) node[xi] {} -- (0,0); %79
	\draw (0,0)  -- (1,1) node[xic] {};
	\draw[kernels2] (0,0) node[not] {} -- (0,1.5) node[xi] {}; }

\DeclareSymbol{Xi4eabisc2r}{1.5}{\draw (-1,2.5) node[xic] {} -- (-1,1) ; \draw[kernels2] (-1,1) node[xi] {} -- (0,0); %79
	\draw (0,0)  -- (1,1) node[xiesf] {};
	\draw[kernels2] (0,0) node[not] {} -- (0,1.5) node[xi] {}; }

\DeclareSymbol{Xi4eabisc3}{1.5}{\draw (-1,2.5) node[xic] {} -- (-1,1) ; \draw[kernels2] (-1,1) node[xic] {} -- (0,0); %80
	\draw (0,0)  -- (1,1) node[xi] {};
	\draw[kernels2] (0,0) node[not] {} -- (0,1.5) node[xi] {}; }

\DeclareSymbol{Xi4eb}{0}{%81
	\draw[kernels2] (0,2) node[xi] {} -- (-1,1) ; \draw[kernels2] (-2,2)  node[xi] {} -- (-1,1) ; \draw (-1,1)  node[not] {} -- (0,0); 
	\draw (0,0) node[xi] {}  -- (1,1) node[xi] {};
}

\DeclareSymbol{Xi4eab}{1.5}{\draw[kernels2] (-1,2.5) node[xi] {} -- (-1,1) ; \draw[kernels2] (-2,2)  node[xi] {} -- (-1,1) ; \draw (-1,1)  node[not] {} -- (0,0); %82
	\draw[kernels2] (0,0)  -- (1,1) node[xi] {};
	\draw[kernels2] (0,0) node[not] {} -- (0,1.5) node[xi] {}; 
}

\DeclareSymbol{Xi4eabdis}{1.5}{\draw[kernels2] (-1,2.5) node[xies] {} -- (-1,1) ; \draw[kernels2] (-2,2)  node[xies] {} -- (-1,1) ; \draw (-1,1)  node[not] {} -- (0,0); %82
	\draw[kernels2] (0,0)  -- (1,1) node[xies] {};
	\draw[kernels2] (0,0) node[not] {} -- (0,1.5) node[xies] {}; 
}

\DeclareSymbol{Xi4eabc1}{1.5}{\draw[kernels2] (-1,2.5) node[xic] {} -- (-1,1) ; \draw[kernels2] (-2,2)  node[xi] {} -- (-1,1) ; \draw (-1,1)  node[not] {} -- (0,0); %83
	\draw[kernels2] (0,0)  -- (1,1) node[xic] {};
	\draw[kernels2] (0,0) node[not] {} -- (0,1.5) node[xi] {}; 
}

\DeclareSymbol{Xi4eabc2}{1.5}{\draw[kernels2] (-1,2.5) node[xi] {} -- (-1,1) ; \draw[kernels2] (-2,2)  node[xi] {} -- (-1,1) ; \draw (-1,1)  node[not] {} -- (0,0); %84
	\draw[kernels2] (0,0)  -- (1,1) node[xic] {};
	\draw[kernels2] (0,0) node[not] {} -- (0,1.5) node[xic] {}; 
}

\DeclareSymbol{Xi4eabbis}{1.5}{\draw[kernels2] (-1,2.5) node[xi] {} -- (-1,1) ; \draw[kernels2] (-2,2)  node[xi] {} -- (-1,1) ; \draw[kernels2] (-1,1)  node[not] {} -- (0,0); %85
	\draw (0,0)  -- (1,1) node[xi] {};
	\draw[kernels2] (0,0) node[not] {} -- (0,1.5) node[xi] {}; 
}

\DeclareSymbol{Xi4eabbisc1}{1.5}{\draw[kernels2] (-1,2.5) node[xic] {} -- (-1,1) ; \draw[kernels2] (-2,2)  node[xi] {} -- (-1,1) ; \draw[kernels2] (-1,1)  node[not] {} -- (0,0); %86
	\draw (0,0)  -- (1,1) node[xic] {};
	\draw[kernels2] (0,0) node[not] {} -- (0,1.5) node[xi] {}; 
}

\DeclareSymbol{Xi4eabbisc1perm}{1.5}{\draw[kernels2] (-1,2.5) node[xi] {} -- (-1,1) ; \draw[kernels2] (-2,2)  node[xic] {} -- (-1,1) ; \draw[kernels2] (-1,1)  node[not] {} -- (0,0); %86
	\draw (0,0)  -- (1,1) node[xic] {};
	\draw[kernels2] (0,0) node[not] {} -- (0,1.5) node[xi] {}; 
}

\DeclareSymbol{Xi4eabbisc2}{1.5}{\draw[kernels2] (-1,2.5) node[xi] {} -- (-1,1) ; \draw[kernels2] (-2,2)  node[xi] {} -- (-1,1) ; \draw[kernels2] (-1,1)  node[not] {} -- (0,0); %87
	\draw (0,0)  -- (1,1) node[xic] {};
	\draw[kernels2] (0,0) node[not] {} -- (0,1.5) node[xic] {}; 
}

\DeclareSymbol{Xi2cbis}{0}{\draw[kernels2] (0,1) -- (0.8,2.2) node[xi] {};\draw[kernels2] (0,-0.25) node[not] {} -- (0,1); \draw[kernels2] (0,1) node[not] {} -- (-0.8,2.2) node[xi] {};}%88

\DeclareSymbol{Xi2cbis1}{0}{\draw (0,1) -- (-0.8,2.2) node[xi] {};\draw[kernels2] (0,-0.25) node[not] {} -- (0,1) node[xi] {}; }%89

\DeclareSymbol{Xi2Xbis}{-2}{\draw[kernels2] (0,-0.25)  -- (-1,1) ; \draw (-1,1) node[xix] {};%91
	\draw[kernels2] (0,-0.25) node[not] {} -- (1,1) node[xi] {};}

\DeclareSymbol{XXi2bis}{-2}{\draw[kernels2] (0,-0.25) -- (-1,1) node[xi] {};%92
	\draw[kernels2] (0,-0.25) node[X] {} -- (1,1) node[xi] {};}

\DeclareSymbol{I1XiIXi}{0}{\draw[kernels2] (0,-0.25) -- (1,1) node[xi] {};%93
	\draw (0,-0.25) node[not] {} -- (-1,1) node[xi] {};}

\DeclareSymbol{I1XiIXib}{0}{\draw  (0,-0.25) node[xi] {} -- (0,1) node[not] {};
	\draw[kernels2] (0,1) -- (0,2.25) ; \draw (0,2.25) node[xi]{}; }%94

\DeclareSymbol{I1XiIXic}{0}{%95
	\draw[kernels2] (0,0) -- (1,1) node[xi] {} ; 
	\draw[kernels2] (0,0) node[not] {}  -- (-1,1) node[not] {} -- (0,2) node[xi] {};
}

\DeclareSymbol{thin}{1.4}{\draw[pagebackground] (-0.3,0) -- (0.3,0); \draw  (0,0) -- (0,2);}
\DeclareSymbol{thin2}{1.4}{\draw[pagebackground] (-0.3,0) -- (0.3,0); \draw[tinydots]  (0,0) -- (0,2);}

\DeclareSymbol{thick}{1.4}{\draw[pagebackground] (-0.3,0) -- (0.3,0); \draw[kernels2]  (0,0) -- (0,2);}

\DeclareSymbol{thick2}{1.4}{\draw[pagebackground] (-0.3,0) -- (0.3,0); \draw[kernels2,tinydots]  (0,0) -- (0,2);}

\DeclareSymbol{Xi4ind}{2}{\draw (0,0) node[xi,label={[label distance=-0.2em]right: \scriptsize  $ i $}]  { } -- (-1,1) node[xi,label={[label distance=-0.2em]left: \scriptsize  $ j $}] {} -- (0,2) node[xi,label={[label distance=-0.2em]right: \scriptsize  $ k $}] {} -- (-1,3) node[xi,label={[label distance=-0.2em]left: \scriptsize  $ \ell $}] {};}%115

\DeclareSymbol{Xi4c1}{2}{\draw (0,0) node[xic] {} -- (-1,1) node[xi] {} -- (0,2) node[xic] {} -- (-1,3) node[xi] {};} %119
\DeclareSymbol{IXi2ex}{0}{\draw (0,-0.25) node[xie] {} -- (-1,1) node[xi] {} ; \draw (0,-0.25)-- (1,1) node[xi] {};}
\DeclareSymbol{IXi2ex1}{0}{\draw (0,-0.25) node[xie] {} -- (-1,1) node[xi] {} -- (0,2) node[xi] {};}

\DeclareSymbol{Xi4b1}{0}{\draw(0,1.5) node[xic] {} -- (0,0); \draw (-1,1) node[xic] {} -- (0,0) node[xi] {} -- (1,1) node[xi] {};}

\DeclareSymbol{Xi4ec1}{0}{\draw (0,2) node[xi] {} -- (-1,1) node[xic] {} -- (0,0) node[xic] {} -- (1,1) node[xi] {};}
\DeclareSymbol{Xi4ec2}{0}{\draw (0,2) node[xic] {} -- (-1,1) node[xi] {} -- (0,0) node[xic] {} -- (1,1) node[xi] {};}
\DeclareSymbol{Xi4ec3}{0}{\draw (0,2) node[xic] {} -- (-1,1) node[xic] {} -- (0,0) node[xi] {} -- (1,1) node[xi] {};}

\DeclareSymbol{I1Xi4ac1}{2}{\draw[kernels2] (0,0) node[not] {} -- (-1,1) ; \draw[kernels2] (0,0) node[not] {} -- (1,1) node[xic] {} ;
	\draw (-1,1) node[xi] {} -- (0,2) node[xic] {} -- (-1,3) node[xi] {};}%161

\DeclareSymbol{I1Xi4ac2}{2}{\draw[kernels2] (0,0) node[not] {} -- (-1,1) ; \draw[kernels2] (0,0) node[not] {} -- (1,1) node[xic] {} ;
	\draw (-1,1) node[xi] {} -- (0,2) node[xi] {} -- (-1,3) node[xic] {};}

\DeclareSymbol{I1Xi4bp}{2}{\draw (0,0) node[not] {} -- (-1,1) node[xi] {} -- (0,2) ; \draw[kernels2] (0,2) node[not] {} -- (-1,3) node[xi] {};\draw[kernels2] (0,2)  -- (1,3) node[xi] {};
}

\DeclareSymbol{I1Xi4bc1}{2}{\draw (0,0) node[xic] {} -- (-1,1) node[xi] {} -- (0,2) ; \draw[kernels2] (0,2) node[not] {} -- (-1,3) node[xi] {};\draw[kernels2] (0,2)  -- (1,3) node[xic] {};
}%165

\DeclareSymbol{I1Xi4bc2}{2}{\draw (0,0) node[xic] {} -- (-1,1) node[xi] {} -- (0,2) ; \draw[kernels2] (0,2) node[not] {} -- (-1,3) node[xic] {};\draw[kernels2] (0,2)  -- (1,3) node[xi] {};
}

\DeclareSymbol{I1Xi4cp}{2}{\draw (0,0) node[not] {} -- (-1,1) node[not] {}; \draw[kernels2] (-1,1) -- (0,2) ; 
	\draw[kernels2] (-1,1) -- (-2,2) node[xi] {} ;
	\draw (0,2) node[xi] {} -- (-1,3) node[xi] {};}

\DeclareSymbol{I1Xi4cc1}{2}{\draw (0,0) node[xic] {} -- (-1,1) node[not] {}; \draw[kernels2] (-1,1) -- (0,2) ; 
	\draw[kernels2] (-1,1) -- (-2,2) node[xi] {} ;
	\draw (0,2) node[xic] {} -- (-1,3) node[xi] {};}%169

\DeclareSymbol{I1Xi4cc2}{2}{\draw (0,0) node[xic] {} -- (-1,1) node[not] {}; \draw[kernels2] (-1,1) -- (0,2) ; 
	\draw[kernels2] (-1,1) -- (-2,2) node[xi] {} ;
	\draw (0,2) node[xi] {} -- (-1,3) node[xic] {};}

\DeclareSymbol{I1Xi4abc1}{2}{\draw[kernels2] (0,0) node[not] {} -- (-1,1) ; \draw[kernels2] (0,0) node[not] {} -- (1,1) node[xic] {};\draw (-1,1) node[xi] {} -- (0,2) ; \draw[kernels2] (0,2) node[not] {} -- (-1,3) node[xic] {};\draw[kernels2] (0,2)  -- (1,3) node[xi] {}; }

\DeclareSymbol{I1Xi4abc2}{2}{\draw[kernels2] (0,0) node[not] {} -- (-1,1) ; \draw[kernels2] (0,0) node[not] {} -- (1,1) node[xic] {};\draw (-1,1) node[xi] {} -- (0,2) ; \draw[kernels2] (0,2) node[not] {} -- (-1,3) node[xi] {};\draw[kernels2] (0,2)  -- (1,3) node[xic] {}; }%173

\DeclareSymbol{R1}{0}{\draw (-1,1) node[xi] {} -- (0,0) node[not] {};%131
	\draw[kernels2] (0,1.5) node[xic] {} -- (0,0) -- (1,1) node[xic] {};}
\DeclareSymbol{R2}{0}{\draw (-1,1) node[xic] {} -- (0,0) node[not] {};
	\draw[kernels2] (0,1.5)  {} -- (0,0) -- (1,1)  {};
	\draw (0,1.5) node[xi] {};
	\draw (1,1) node[xic] {};
}%132
\DeclareSymbol{R3}{1}{\draw[kernels2] (-1,1.5)  {} -- (0,0) node[not] {} -- (1,1.5);%133
	\draw (-1,1.5) node[xi] {};
	\draw[kernels2] (0,3) {} -- (1,1.5) -- (2,3)  {};
	\draw  (0,3) node[xic] {} ;
	\draw (2,3) node[xic] {};}
\DeclareSymbol{R4}{1}{\draw[kernels2] (-1,1.5) node[xic] {} -- (0,0) node[not] {} -- (1,1.5);%134
	\draw[kernels2] (0,3) {} -- (1,1.5) -- (2,3) node[xic] {};
	\draw (0,3) node[xi] {};}

\DeclareSymbol{I1Xi4bcp}{2}{\draw (0,0) node[not] {} -- (-1,1) node[not] {}; \draw[kernels2] (-1,1) -- (0,2) ; %175
	\draw[kernels2] (-1,1) -- (-2,2) node[xi] {} ; \draw[kernels2] (0,2) node[not] {} -- (-1,3) node[xi] {};\draw[kernels2] (0,2)  -- (1,3) node[xi] {};
}

\DeclareSymbol{I1Xi4bcc1}{2}{\draw (0,0) node[xic] {} -- (-1,1) node[not] {}; \draw[kernels2] (-1,1) -- (0,2) ; 
	\draw[kernels2] (-1,1) -- (-2,2) node[xi] {} ; \draw[kernels2] (0,2) node[not] {} -- (-1,3) node[xi] {};\draw[kernels2] (0,2)  -- (1,3) node[xic] {};
}

\DeclareSymbol{I1Xi4bcc2}{2}{\draw (0,0) node[xic] {} -- (-1,1) node[not] {}; \draw[kernels2] (-1,1) -- (0,2) ; 
	\draw[kernels2] (-1,1) -- (-2,2) node[xi] {} ; \draw[kernels2] (0,2) node[not] {} -- (-1,3) node[xic] {};\draw[kernels2] (0,2)  -- (1,3) node[xi] {};
} %177

\DeclareSymbol{2I1Xi4bc1}{2}{\draw[kernels2] (0,0) node[not] {} -- (-1,1) ;
	\draw[kernels2] (0,0) -- (1,1);
	\draw (-1,1) node[xic] {} -- (-1,2.5) node[xi] {};
	\draw (1,1)  node[xic] {} -- (1,2.5) node[xi] {};
}

\DeclareSymbol{2I1Xi4bc2}{2}{\draw[kernels2] (0,0) node[not] {} -- (-1,1) ;
	\draw[kernels2] (0,0) -- (1,1);
	\draw (-1,1) node[xi] {} -- (-1,2.5) node[xic] {};
	\draw (1,1)  node[xic] {} -- (1,2.5) node[xi] {};
}

\DeclareSymbol{diff2I1Xi4bc2}{2}{\draw (0,0) node[diff] {} -- (-1,1) ;
	\draw (0,0) -- (1,1);
	\draw (-1,1) node[xi] {} -- (-1,2.5) node[xic] {};
	\draw (1,1)  node[xic] {} -- (1,2.5) node[xi] {};
}

\DeclareSymbol{2I1Xi4bc3}{2}{\draw[kernels2] (0,0) node[not] {} -- (-1,1) ;
	\draw[kernels2] (0,0) -- (1,1);
	\draw (-1,1) node[xic] {} -- (-1,2.5) node[xic] {};
	\draw (1,1)  node[xi] {} -- (1,2.5) node[xi] {};
}

\DeclareSymbol{Xi41}{0}{\draw (0,1) -- (0.8,2.2) node[xic] {};\draw (0,-0.25) node[xi] {} -- (0,1) node[xi] {} -- (-0.8,2.2) node[xic] {};} 

\DeclareSymbol{Xi42}{0}{\draw (0,1) -- (0.8,2.2) node[xi] {};\draw (0,-0.25) node[xic] {} -- (0,1) node[xi] {} -- (-0.8,2.2) node[xic] {};}

\DeclareSymbol{Xi4ca1}{0}{\draw (0,1) -- (-1,2.2) node[xic] {};\draw (0,-0.25) node[xi] {} -- (0,1) ; \draw[kernels2] (0,1) node[not] {} -- (1,2.2) node[xic] {};
	\draw[kernels2] (0,1) {} -- (0,2.7) node[xi] {};
}

\DeclareSymbol{Xi4ca2}{0}{\draw (0,1) -- (-1,2.2) node[xi] {};\draw (0,-0.25) node[xi] {} -- (0,1) ; \draw[kernels2] (0,1) node[not] {} -- (1,2.2) node[xic] {};
	\draw[kernels2] (0,1) {} -- (0,2.7) node[xic] {};
}

\DeclareSymbol{Xi4cap}{0}{\draw (0,1) -- (-1,2.2) node[xi] {};\draw (0,-0.25) node[not] {} -- (0,1) ; \draw[kernels2] (0,1) node[not] {} -- (1,2.2) node[xi] {};
	\draw[kernels2] (0,1) {} -- (0,2.7) node[xi] {};
}

\DeclareSymbol{Xi3a}{0}{
	\draw (-1,1)  node[xi] {} -- (0,0); 
	\draw (0,0) node[xi] {}  -- (1,1) node[xi] {};
}

\DeclareSymbol{Xi4ebc1}{0}{
	\draw[kernels2] (0,2) node[xi] {} -- (-1,1) ; \draw[kernels2] (-2,2)  node[xic] {} -- (-1,1) ; \draw (-1,1)  node[not] {} -- (0,0); 
	\draw (0,0) node[xic] {}  -- (1,1) node[xi] {};
}

\DeclareSymbol{Xi4ebc2}{0}{
	\draw[kernels2] (0,2) node[xi] {} -- (-1,1) ; \draw[kernels2] (-2,2)  node[xi] {} -- (-1,1) ; \draw (-1,1)  node[not] {} -- (0,0); 
	\draw (0,0) node[xic] {}  -- (1,1) node[xic] {};
}

\DeclareSymbol{Xi2cbispex}{0}{\draw[kernels2] (0,1) -- (0.8,2.2) node[xi] {};\draw (0,-0.25) node[xie] {} -- (0,1); \draw[kernels2] (0,1) node[not] {} -- (-0.8,2.2) node[xi] {};}

\DeclareSymbol{Xi2cbis1p}{0}{\draw (0,1) -- (-0.8,2.2) node[xi] {};\draw (0,-0.25) node[not] {} -- (0,1) node[xi] {}; }

\DeclareSymbol{Xi2Xp}{-2}{\draw (0,-0.25) node[not] {} -- (-1,1) node[xix] {};} % 229 not used

\DeclareSymbol{I1XiIXib}{0}{\draw  (0,-0.25) node[xi] {} -- (0,1) node[not] {};
	\draw[kernels2] (0,1) -- (0,2.25) ; \draw (0,2.25) node[xi]{}; }

\DeclareSymbol{IXi2b}{0}{\draw  (0,-0.25) node[xi] {} -- (0,1) node[not] {};
	\draw (0,1) -- (0,2.25) ; \draw (0,2.25) node[xi]{}; }

\DeclareSymbol{IXi2bex}{0}{\draw  (0,-0.25) node[xi] {} -- (0,1) node[xie] {};
	\draw (0,1) -- (0,2.25) ; \draw (0,2.25) node[xi]{}; }

\def\1{\mathbf{\symbol{1}}}

\DeclareSymbol{diff}{0}{
	\draw (0,0.5) node[diff] {};
}

\DeclareSymbol{diff1}{0}{
	\draw (0,0.5) node[diff1] {};
}

\DeclareSymbol{diff2}{0}{
	\draw (0,0.5) node[diff2] {};
}

\DeclareSymbol{geo}{0}{
	\draw (0,0) node[diff] {};
	\draw (0.3,0) node[diff] {};
}

\DeclareSymbol{generic}{0}{
	\draw (0,0.6) node[xi] {};
}

\DeclareSymbol{g}{0}{
	\draw (0,0.6) node[g] {};
}

\DeclareSymbol{Ito}{0}{
	\draw (0,0.6) node[xies] {};
}

\DeclareSymbol{Itob}{0}{
	\draw (0,0.6) node[xiesf] {};
}

\DeclareSymbol{greycirc}{0}{
	\draw (0,0.3) node[xi] {};
}

\DeclareSymbol{not}{0}{
	\draw (0,0.6) node[not] {};
	\draw[tinydots] (0,0.6) circle (0.8);
}

\DeclareSymbol{genericb}{0}{
	\draw (0,0.6) node[xic] {};
}

\DeclareSymbol{bluecirc}{0}{
	\draw (0,0.3) node[xic] {};
}

\DeclareSymbol{genericxix}{0}{
	\draw (0,0.6) node[xix] {};
}

\DeclareSymbol{genericX}{0}{
	\draw (0,0.6) node[X] {};
}

\DeclareSymbol{diffIto}{1}{
	\draw  (0,2.5) -- (0,0) ;
	\draw (0,-0.1) node[diff] {};
	\draw (0,2.5) node[xies] {};
}
\DeclareSymbol{Itodiff}{2}{
	\draw(0,2.9) -- (0,-0.2);
	\draw (0,2.9) node[diff] {};
	\draw (0,-0.1) node[xies] {};
}

\DeclareSymbol{diffgeneric}{1}{
	\draw  (0,2.5) -- (0,0) ;
	\draw (0,-0.1) node[diff] {};
	\draw (0,2.5) node[xi] {};
}

\DeclareSymbol{genericdiff}{2}{
	\draw(0,2.9) -- (0,-0.2);
	\draw (0,2.9) node[diff] {};
	\draw (0,-0.1) node[xi] {};
}

\DeclareSymbol{diffdot}{2}{
	\draw  (0,3) -- (0,-0.1) ;
	\draw (0,3) node[not] {};
	\draw (0,-0.1) node[diff] {};
}

\DeclareSymbol{diffdotmini}{0}{
	\draw  (0,0) -- (0,1.2) ;
	\draw (0,1.2) node[not] {};
	\draw (0,0) node[diffmini] {};
}

\DeclareSymbol{dotdiff}{2}{
	\draw[kernelsmod]  (0,3) -- (0,-0.1) ;
	\draw (0,3) node[diff] {};
	\draw (0,-0.1) node[not] {};
}

\DeclareSymbol{dotdiff1}{2}{
	\draw[kernelsmod]  (0,3) -- (0,-0.1) ;
	\draw (0,3) node[diff1] {};
	\draw (0,-0.1) node[not] {};
}

\DeclareSymbol{dotdiff1mini}{0}{
	\draw[kernelsmod]  (0,1.2) -- (0,0) ;
	\draw (0,1.2) node[diffmini] {};
	\draw (0,0) node[not] {};
}

\DeclareSymbol{dotdiff2}{2}{
	\draw (0,3) -- (0,-0.1) ;
	\draw (0,3) node[diff] {};
	\draw (0,-0.1) node[not] {};
}

\DeclareSymbol{dotdiff2mini}{0}{
	\draw (0,1.2) -- (0,0) ;
	\draw (0,1.2) node[diffmini] {};
	\draw (0,0) node[not] {};
}

\DeclareSymbol{dotdiffstraight}{0}{
	\draw  (0,3) -- (0,-0.1) ;
	\draw (0,3) node[diff] {};
	\draw (0,-0.1) node[not] {};
}

\DeclareSymbol{arbre1}{0}{
	\draw  (0,0) -- (1.5,1.5) ;
	\draw (1.5,1.5) node[not] {};
	\draw (0,0) node[not] {};
}

\DeclareSymbol{arbre2}{0}{
	\draw  (0,0) -- (1.5,1.5) ;
	\draw[kernelsmod] (0,0) -- (-1.5,1.5);
	\draw (1.5,1.5) node[not] {};
	\draw (0,0) node[not] {};
	\draw (-1.5,1.5) node[xi] {};
}

\DeclareSymbol{arbre3}{0}{
	\draw  (0,0) -- (1.5,1.5) ;
	\draw[kernelsmod] (1.5,1.5) -- (0,3);
	\draw (0,0) node[not] {};
	\draw (1.5,1.5) node[not] {};
	\draw (0,3) node[xi] {};
}

\DeclareSymbol{treeeval}{0}{
	\draw (0,0) -- (1,1);
	\draw (0,0) node[xi] {};
	\draw (1.25,1.25) node[xi] {};
	\draw (-0.6,0.6) node[]{\tiny{$i$}};
	\draw (0.65,1.85) node[]{\tiny{$j$}};
}

\DeclareSymbol{testeval}{0}{
	\draw (0,0) -- (1,1);
	\draw (0,0) -- (-1,1);
	\draw (0,0) node[xi] {};
	\draw (1.25,1.25) node[xi] {};
	\draw (-1.25,1.25) node[xi] {};
	\draw (-0.6,-0.6) node[]{\tiny{$i$}};
	\draw (0.65,1.85) node[]{\tiny{$j$}};
	\draw (-1.95,1.85) node[]{\tiny{$k$}};
}

\DeclareSymbol{treeeval2}{0}{
	\draw[kernelsmod] (-0.25,-1) -- (1,0.5) ;
	\draw[kernelsmod] (1,0.5) -- (-0.25,2);
	\draw (1,0.5) node[diff2] {};
	\draw (-0.25,-1) node[not] {};
	\draw (-0.25,2) node[xi] {};
	\draw (-0.6,1.2) node[]{\tiny{1}};
}

\DeclareSymbol{arbreact}{1}{
	\draw (0,0) node[not] {};
	\draw[kernelsmod] (0,0) -- (1,1);
	\draw[kernelsmod] (0,0) -- (-1,1);
	\draw (-1,1) node[xic] {};
	\draw  (0,2) -- (1,1) ;
	\draw (0,2) node[xic] {};
	\draw (1,1) node[xi] {};
}

\DeclareSymbol{arbreact1}{0}{
	\draw (0,-1.5) -- (0,0);
	\draw[kernelsmod] (0,0) -- (1,1);
	\draw[kernelsmod] (0,0) -- (-1,1);
	\draw  (0,2) -- (1,1) ;
	\draw (0,-1.5) node[diff] {};
	\draw (0,0) node[not] {};
	\draw (-1,1) node[xic] {};
	\draw (0,2) node[xic] {};
	\draw (1,1) node[xi] {};
}

\DeclareSymbol{arbreact2}{0}{
	\draw (0,-0.75) -- (-1,0.5); 
	\draw (0,-0.75) -- (1,0.5);
	\draw (0,1.5) -- (1,0.5);
	\draw (0,1.5) node[xic] {};
	\draw (1,0.5) node[xi] {};
	\draw (-1,0.5) node[xic] {};
	\draw (0,-0.75) node[diff] {};
}

\DeclareSymbol{arbreact3}{0}{
	\draw[kernelsmod] (0,-0.75) -- (-1,0.5); 
	\draw[kernelsmod] (0,-0.75) -- (1,0.5);
	\draw (0,1.75) -- (1,0.5);
	\draw (2,1.75) -- (1,0.5);
	\draw (0,1.75) node[xic] {};
	\draw (1,0.5) node[diff] {};
	\draw (-1,0.5) node[xic] {};
	\draw (2,1.75) node[xi] {};
	\draw (0,-0.75) node[not] {};
}

\DeclareSymbol{pre_im_I1Xitwo}{0}{
	\draw[kernels2] (0,-0.3) node[not] {} -- (-0.6,0.7) ;
	\draw[kernels2] (0,-0.3) -- (0.6,0.7);
	\draw (0,0.9) node[g] {};
}

\DeclareSymbol{pre_im_cI1Xi4ab}{2}{
	\draw[kernels2] (0,-1) node[not] {} -- (-0.6,0) ;
	\draw[kernels2] (0,-1) -- (0.6,0);
	\draw (0,0.2) node[g] {};
	\draw (0,0.6) -- (0,1.5);
	\draw[kernels2] (0,1.5) node[not] {} -- (-0.6,2.5) ;
	\draw[kernels2] (0,1.5) -- (0.6,2.5);
	\draw (0,2.7) node[g] {};
}%46

\DeclareSymbol{pre_im_I1Xi4acc2}{0}{
	\draw[kernels2] (-1,-0.5) node[not] {} -- (-1.6,0.5) ;
	\draw[kernels2] (-1,-0.5) -- (-0.4,0.5);
	\draw[kernels2] (-1,-0.5) -- (0.2,-1.5) node[not] {} ;
	\draw (-1,1.1) -- (-1,2);
	\draw[kernels2] (0.2,-1.5) -- (0.2,2);
	\draw (-1,0.7) node[g] {};
	\draw (-0.3,2.2) node[g] {};
}

\DeclareSymbol{pre_im_I1Xi4abcc2}{2}{
	\draw[kernels2] (0,-1) node[not] {} -- (-1,0) node[not] {};
	\draw[kernels2] (-1,1.2) node[not] {} -- (-1,0);
	\draw[kernels2] (-1,1.2) -- (-1.5,2.5);
	\draw[kernels2] (-1,1.2) -- (-0.5,2.5);
	\draw[kernels2] (-1,0) -- (0.7,2.5);
	\draw[kernels2] (0,-1) -- (1.5,2.5);
	\draw (-1,2.7) node[g] {};
	\draw (1,2.7) node[g] {};
}

\DeclareSymbol{pre_im_2I1Xi4c1}{2}{
	\draw[kernels2] (0,-0.5) node[not] {} -- (-1,0.5) node[not] {};
	\draw[kernels2] (0,-0.5) -- (1,0.5) node[not] {};
	\draw[kernels2] (-1,0.5) node[not] {}-- (-1.7,2);
	\draw[kernels2]  (-1,2) -- (1,0.5);
	\draw[kernels2] (-1,0.5) -- (1,2);
	\draw[kernels2] (1,0.5) -- (1.7,2);
	\draw (-1.2,2.2) node[g] {};
	\draw (1.2,2.2) node[g] {};
}

\DeclareSymbol{pre_im_Xi4eabisc2}{2}{
	\draw[kernels2] (1.2,-0.5) node[not] {} -- (-0.7,0.8) ;
	\draw[kernels2] (1.2,-0.5) -- (0.4,0.8);
	\draw (0,1.4)  -- (0,2.2);
	\draw (1.2,2.2) -- (1.2,-0.6);
	\draw (0,1) node[g] {};
	\draw (0.6,2.4) node[g] {};
}

\DeclareSymbol{pre_im_Xi4eabisc22}{2}{
	\draw (1.2,-0.5) node[not] {} -- (-0.7,0.8) ;
	\draw[kernels2] (1.2,-0.5) -- (0.4,0.8);
	\draw (0,1.4)  -- (0,2.2);
	\draw[kernels2] (1.2,2.2) -- (1.2,-0.6);
	\draw (0,1) node[g] {};
	\draw (0.6,2.4) node[g] {};
}

\DeclareSymbol{pre_im_Xi4eabisc222}{2}{
	\draw[kernels2] (0.4,-0.5) node[not] {} -- (-0.6,1) ;
	\draw[kernels2] (1.2,0) -- (0.3,1);
	\draw (0,1.1)  -- (0,2.5);
	\draw[kernels2] (1.2,2.5) -- (1.2,0) node[not] {} -- (0.4,-0.6);
	\draw (0,1.2) node[g] {};
	\draw (0.6,2.5) node[g] {};
}

\DeclareSymbol{pre_im_Xi4eabc2}{2}{
	\draw (0,-0.5) node[not] {} -- (-1,0.5) node[not] {};
	\draw[kernels2] (-1,0.5) -- (-1.5,2);
	\draw[kernels2] (0,-0.5)  -- (0.7,2);
	\draw[kernels2] (-1,0.5) -- (-0.5,2);
	\draw[kernels2] (0,-0.5) -- (1.5,2);
	\draw (-1,2.2) node[g] {};
	\draw (1,2.2) node[g] {};
}

\DeclareSymbol{pre_im_Xi4eabbisc2}{2}{
	\draw[kernels2] (0,-0.5) node[not] {} -- (-1,0.5) node[not] {};
	\draw[kernels2] (-1,0.5) -- (-1.5,2);
	\draw[kernels2] (0,-0.5)  -- (0.7,2);
	\draw[kernels2] (-1,0.5) -- (-0.5,2);
	\draw (0,-0.5) -- (1.5,2);
	\draw (-1,2.2) node[g] {};
	\draw (1,2.2) node[g] {};
}

\DeclareSymbol{pre_im_I1Xi4abcc1}{2}{
	\draw[kernels2] (0,-1) node[not] {} -- (-1,0) node[not] {};
	\draw[kernels2] (0,1.1) node[not] {} -- (-1,0);
	\draw[kernels2] (-1,0) -- (-1.5,2.5);
	\draw[kernels2] (0,1.1) node[not] {} -- (-0.5,2.5);
	\draw[kernels2] (0,1.1) -- (0.5,2.5);
	\draw[kernels2] (0,-1) -- (1.5,2.5);
	\draw (-1,2.7) node[g] {};
	\draw (1,2.7) node[g] {};
}

\DeclareSymbol{pre_im_Xi4eabc1}{2}{
	\draw (0,-0.5) node[not] {} -- (-1,0.5) node[not] {};
	\draw[kernels2] (-1,0.5) -- (-1.5,2);
	\draw[kernels2] (0,-0.5)  -- (-0.5,2);
	\draw[kernels2] (-1,0.5) -- (0.8,2);
	\draw[kernels2] (0,-0.5) -- (1.5,2);
	\draw (-1,2.2) node[g] {};
	\draw (1,2.2) node[g] {};
}

\DeclareSymbol{pre_im_Xi4ba1b}{2}{
	\draw[kernels2] (0,0) node[not] {}  -- (1.8,1.5);
	\draw[kernels2] (0,0) -- (0.8,1.5);
	\draw (0,-0.1) -- (-1.8,1.5);
	\draw (0,-0.1) -- (-0.8,1.5);
	\draw (-1,1.7) node[g] {};
	\draw (1,1.7) node[g] {};
}

\DeclareSymbol{pre_im_Xi4ba2}{2}{
	\draw (0,-0.1) node[not] {}  -- (1.8,1.5);
	\draw[kernels2] (0,0) -- (0.8,1.5);
	\draw (0,-0.1) -- (-1.8,1.5);
	\draw[kernels2] (0,0) -- (-0.8,1.5);
	\draw (-1,1.7) node[g] {};
	\draw (1,1.7) node[g] {};
}

\DeclareSymbol{pre_im_Xi4cabc2}{2}{
	\draw[kernels2] (0,-0.5) node[not] {} -- (-1,0.5) node[not] {};
	\draw[kernels2] (-1,0.5) -- (-1.5,2);
	\draw (-1,0.5)  -- (0.7,2);
	\draw[kernels2] (-1,0.5) -- (-0.5,2);
	\draw[kernels2] (0,-0.5) -- (1.5,2);
	\draw (-1,2.2) node[g] {};
	\draw (1,2.2) node[g] {};
}

\DeclareSymbol{pre_im_Xi4cabc1}{2}{
	\draw[kernels2] (0,-0.5) node[not] {} -- (-1,0.5) node[not] {};
	\draw (-1,0.5) -- (-1.5,2);
	\draw[kernels2] (-1,0.5)  -- (0.7,2);
	\draw[kernels2] (-1,0.5) -- (-0.5,2);
	\draw[kernels2] (0,-0.5) -- (1.5,2);
	\draw (-1,2.2) node[g] {};
	\draw (1,2.2) node[g] {};
}

\DeclareSymbol{pre_im_Xi4eabbisc1}{2}{
	\draw[kernels2] (0,-0.5) node[not] {} -- (-1,0.5) node[not] {};
	\draw[kernels2] (-1,0.5) -- (-1.5,2);
	\draw[kernels2] (0,-0.5)  -- (-0.5,2);
	\draw[kernels2] (-1,0.5) -- (0.8,2);
	\draw (0,-0.5) -- (1.5,2);
	\draw (-1,2.2) node[g] {};
	\draw (1,2.2) node[g] {};
}

\DeclareSymbol{pre_im_1}{0}{
	\draw[kernels2] (0,-0.5) node[not] {} -- (-0.6,0.5) ;
	\draw[kernels2] (0,-0.5) -- (0.6,0.5);
	\draw (0,1.1)  -- (-0.55,2);
	\draw (0,1.1)  -- (0.55,2);
	\draw (0,0.7) node[g] {};
	\draw (0,2.2) node[g] {};
}

\DeclareSymbol{disconnect}{0}{
	\draw[kernels2] (0,-0.5) node[not] {} -- (-0.6,0.5) ;
	\draw[kernels2] (0,-0.5) -- (0.6,0.5);
	\draw (-0.55,1.1)  -- (-0.55,2.3);
	\draw (0.55,2.3) -- (0.55,1.5) -- (1.2,1.5) -- (1.2,3.5) -- (0.55,3.5) -- (0.55,2.7);
	\draw (0,0.7) node[g] {};
	\draw (0,2.5) node[g] {};
}

\DeclareSymbol{pre_im_2}{2}{\draw[kernels2] (0,0) node[not] {} -- (-1,1) node[not] {};
	\draw[kernels2] (0,0) -- (1,1) node[not] {};
	\draw[kernels2] (-1,1) -- (-1.5,2.5);
	\draw[kernels2] (-1,1) -- (-0.5,2.5);
	\draw[kernels2] (1,1) -- (0.5,2.5);
	\draw[kernels2] (1,1) -- (1.5,2.5);
	\draw (-1,2.7) node[g] {};
	\draw (1,2.7) node[g] {};
}

\DeclareSymbol{CX_rec}{0}{
	\draw [black] (-0.3,1) to (-0.3,-0.3);
	\draw [black] (0.3,1) to (0.3,-0.3);
	\draw [black] (-0.3,1) to (-0.3,2.3);
	\draw [black] (0.3,1) to (0.3,2.3);
	\draw (0,1) node[rec] {};
}

\DeclareSymbol{CX_cerc}{0}{
	\draw [black] (0,1) to (0,-0.3);
	\draw (0,1) node[cerc] {};
}

\DeclareMathAlphabet{\mathpzc}{OT1}{pzc}{m}{it}

\usepackage{pifont}

\let\d\partial

\def\eqref#1{(\ref{#1})}

\def\Hom{\mathop{\mathrm{Hom}}\nolimits}

\makeatletter % Stolen from the internet to make a fat \cdot which isn't as fat as a \bullet
\newcommand*{\bigcdot}{}% Check if undefined
\DeclareRobustCommand*{\bigcdot}{%
	\mathbin{\mathpalette\bigcdot@{}}%
}
\newcommand*{\bigcdot@scalefactor}{.5}
\newcommand*{\bigcdot@widthfactor}{1.15}
\newcommand*{\bigcdot@}[2]{%
	\sbox0{$#1\vcenter{}$}% math axis
	\sbox2{$#1\cdot\m@th$}%
	\hbox to \bigcdot@widthfactor\wd2{%
		\hfil
		\raise\ht0\hbox{%
			\scalebox{\bigcdot@scalefactor}{%
				\lower\ht0\hbox{$#1\bullet\m@th$}%
			}%
		}%
		\hfil
	}%
}
\makeatother

\tcbset
{colframe=boxcolor,colback=symbols!7!pagebackground,coltext=pageforeground,
	fonttitle=\bfseries,nobeforeafter,center title,size=fbox,boxsep=1.5pt,
	top=0mm,bottom=0mm,boxsep=0mm,tcbox raise base}

\def\two{{\<generic>\kern0.05em\<genericb>}}
\def\twoI{{\<Ito>\kern0.05em\<Itob>}}

\def\mail#1{\burlalt{#1}{mailto:#1}}

\DeclareRobustCommand{\rchi}{{\mathpalette\irchi\relax}}
\newcommand{\irchi}[2]{\raisebox{\depth}{$#1\chi$}}

\usepackage{thmtools}

\begin{document}

\title{Elementary differentials from multi-indices to rooted trees}

\author{Yvain Bruned, Paul Laubie}
\institute{ 
 Universite de Lorraine, CNRS, IECL, F-54000 Nancy, France
  \\
Email:\ \begin{minipage}[t]{\linewidth}
\mail{yvain.bruned@univ-lorraine.fr}
\\ \mail{paul.laubie@univ-lorraine.fr}
\end{minipage}}

\maketitle 

\begin{abstract}
Rooted trees are essential for describing numerical schemes via the so-called B-series. They have also been  used extensively in rough analysis for expanding solutions
of singular Stochastic Partial Differential Equations (SPDEs). When one considers scalar-valued equations, the most efficient combinatorial set is multi-indices. In this paper, we investigate the existence of intermediate combinatorial sets that will lie between multi-indices and rooted trees. We provide a negative result stating that there is no combinatorial set encoding elementary differentials in dimension $d\neq 1$, and compatible with the rooted trees and the multi-indices aside from the rooted trees. This does not close the debate of the existence of such combinatorial sets, but it shows that it cannot be obtained via a naive and natural approach. 
\end{abstract}
\setcounter{tocdepth}{2}
\setcounter{secnumdepth}{4}
\tableofcontents

\section{Introduction}

Since the 60's and the pioneer work of Butcher \cite{Butcher72}, rooted trees  have been the  main combinatorial object for describing numerical schemes of ODEs of the form:
\begin{equs}
	y'(t)  = f(y(t)), \quad y(0) = y_0.
\end{equs}
The scheme is described via a B-series (terminology coming from \cite{HW74}) which are series of the form
\begin{equation}
	\label{eq:Bseries}
	B(\alpha, h,f,y) 
	=\sum_{\tau \in T} \frac{h^{|\tau|}\alpha(\tau)}{S(\tau)} F_f(\tau)(y).
\end{equation}
Where $T$ are rooted trees, $h$ is the time step of the method, $|\tau|$ corresponds to the number of nodes of $\tau$, $S(\tau)$ is a symmetry factor, $\alpha(\tau)$ are the coefficients of the method, and $F(\tau)$ are the elementary differentials. This formalism has been used beyond numerical analysis in the context of rough analysis, where one wants to deal with a rough equation of the form
\begin{equs}
	d Y_t  = \sum_{i=1}^d f_i(Y_t) d X^i_t, \quad Y_0 = y_0.
\end{equs}
Where the $X^i$ are continuous paths not differentiable. If one has a rough path above $X$, which is a collection of iterated integrals built out of the $X^i$, one can look at solutions of the form of a B-series:
\begin{equs} \label{rough_path_expansion}
		Y_t
	=\sum_{\tau \in T_d} \frac{F_f(\tau)(Y_s)}{S(\tau)}  \mathbb{X}_{s,t}(\tau).
\end{equs}
Where now $T_d$ are decorated trees with decorations on the node, and $ \mathbb{X}_{s,t}$ is the rough path, also called branched rough paths. This tree series expansion has been introduced in \cite{Gub06}, where one does not have some geometric properties (integration by parts available) for the iterated integrals. For geometric rough paths, see \cite{lyons1998,Gubinelli2004}, and for the link between the two types of rough paths, see \cite{HK15}. 
The idea of the expansion \eqref{rough_path_expansion} has been used for building up Regularity Structures in \cite{reg}, which proposes a new type of Taylor expansion for solutions of singular Stochastic Partial Differential Equations (SPDEs). It uses decorated trees with decorations both on the edges and the nodes \cite{BHZ}, and one has a new B-series formalism for the first time described in \cite{BCCH}. Decorated trees and their Hopf algebra structures are one of the main components of this theory.

Recently, a new combinatorial set called multi-indices has been proposed  in \cite{OSSW} for scalar-valued SPDEs. It enjoys similar Hopf algebras as in the context of decorated trees \cite{LOT}, and one can work with some extensions for dealing with more than one noise in \cite{BL24,BD23}. Multi-indices are connected to Novikov algebras as explained in \cite{DL}.
The idea is to notice that for scalar-valued equations, some trees have the same elementary differential, and therefore, one wants to gather them under the umbrella of the same combinatorics element. One wants to avoid over-parametrisation. For multi-indices, this boils down to keeping for each node of a tree its arity. 

Shortly after the previous works on SPDEs, it was noticed that multi-indices naturally occur in the context of rough paths \cite{Li23,BBH25} and numerical analysis \cite{MV16,BHE24}, where the multi-indices B-series uniquely characterise the Taylor expansion of one-dimensional local and affine equivariant maps. These combinatorics questions also appear in the other approaches developed for singular SPDEs, such as para-controlled calculus \cite{GIP15,BB19}, and the flow approach \cite{Duc21,CF24,BM25}. It is also a fundamental question in dispersive PDEs for low regularity schemes \cite{BS}, derivation of kinetic wave equations \cite{DH23} and subcritical equations with random initial data \cite{DNY22}.
Let us conclude by noting that the study of symmetries, such as the chain rule \cite{BGHZ,BD24,BB24}, varies according to the chosen combinatorial sets.

The main idea of this work is to investigate the potential existence of new combinatorial sets lying between rooted trees and multi-indices that could encode elementary differentials in small dimensions. As we want something easy to write, to manipulate, and with good algebraic properties, we want this set to be compatible with the rooted trees and the multi-indices. One has a natural surjection from the rooted trees to the multi-indices. One wants to know if it is possible to find an intermediate combinatorial set $A$ that could encode elementary differentials in small dimensions, and such that one of the two below is true
\begin{equs}
	\text{Rooted Trees} \, \twoheadrightarrow A, \quad A \twoheadrightarrow  \text{Multi-indices}.
\end{equs}
Hopefully, we want this intermediate set to satisfy nice algebraic properties. In this work, we provide negative results on these identities summarised in the theorem below:

\begin{theorem}
	There is no combinatorial set encoding elementary differentials in dimension $d\neq 1$ and compatible with the rooted trees and the multi-indices aside from the rooted trees.
\end{theorem}
This negative answer does not necessarily close the debate about the existence of a new type of coding in small dimensions. It states that if one starts with a naive and natural approach, one cannot get a new set. 

Although we do not use any technical definition from the algebraic side, and most proofs rely on elementary group theory, we should give some insight about the algebraic context in which this article sits.
The algebraic approach we follow in this article is virtually the same as in \cite{DL25}, using combinatorial species (see \cite{BLL,J81} for an introduction to the theory of species). 
Due to this approach, although never explicitly mentioned, operad theory is omnipresent in this article (see \cite{LV} for an introduction to the theory of algebraic operads). 
This is especially visible with the use of the theory of polynomial identities in Witt algebras \cite{D00,KU16}. 
Indeed, it is known that the study of polynomial identities in characteristic $0$ exactly correspond to the study of operadic ideals.
In particular, each theorem of this article may be stated in the language of algebraic operads.
Fortunately, since none of the results nor their proof require technical definitions from the theory of species or the theory of algebraic operads, no a priori knowledge in those topics is required to read the article.

Let us  briefly explain what is happening. We start in Section \ref{Sec::2} by defining elementary differentials on rooted trees with $n$ nodes labelled from $1$ to $n$ (see Definition \ref{elementary_diff}), denoted for a rooted tree $\tau$ by
\begin{equs}
		F(\tau)\left((f^i)_{i\in \{1,\dots,n\}}\right), \quad f^i \in \mathcal{C}^\infty(\mathbb{R}^d,\mathbb{R}^d).
\end{equs}
Let us stress that one chooses a different non-linearity $f$ for each node of $\tau$, which gives a lot of freedom. We do the same for multi-indices. For a multi-index $ z^{\beta}$, one considers
\begin{equs}
		G(z^{\beta})\left(f_1,...,f_n\right), \quad f_i \in  \mathcal{C}^\infty(\mathbb{R},\mathbb{R})
\end{equs}
which is defined in \eqref{elementary_diff_multi}. For $d=1$, $F$ and $G$ are connected in Proposition \ref{prp:multi_trees} via the surjection $\pi$ that allows  moving from the rooted trees to the multi-indices.
Inspired by this property, we define what is a combinatorial description of the elementary differentials compatible with the rooted trees and the multi-indices in Definition \ref{Def_combinatorial}.

In Section \ref{Sec::3}, we state our first negative result, Theorem \ref{thm:1}, showing that for $d \neq 1$, the only combinatorial description compatible with rooted trees are the rooted trees. The proof relies on the crucial identity from Proposition \ref{prp:geofixed}; for a rooted tree $\tau$, we have
\begin{equs}
	F(\sigma . \tau) = F(\tau)  \quad \Longleftrightarrow \quad \sigma . \tau = \tau.
\end{equs}
where $\sigma$ is a bijection of the nodes of $\tau$ acting on $\tau$ and producing another tree with the same nodes (see Definition \ref{def_sigma}). One should be especially careful with the notation $\sigma.\tau$ since this is not the usual action on rooted trees by permutation of the vertices. The second point is to use the multi-indices in Proposition \ref{prp:transitiveaction} to say if $\tau_1$ and $\tau_2$ have the same multi-indices, one has the existence of $\sigma$ such that $ \sigma . \tau_1 = \tau_2 $.

In Section \ref{Sec::4}, we state our second negative result, Theorem \ref{thm:2}, which establishes the non-existence of a faithful combinatorial description compatible with the multi-indices. 
Its proof relies on the knowledge of minimal left identities (see Theorem \ref{minimal_identity}) for the Witt algebra of rank $d$ (see Definition \ref{def_Witt}). 
Then, one connects this identity with linear rooted trees in Proposition \ref{correspondance_left}. 
For finishing the proof, one has to compute the character associated with the elementary differentials arising from linear trees (see Corollary \ref{representation_d}).
We provide a stronger negative result in dimension $2$ (see Theorem \ref{thm:3}), saying that there is not a faithful combinatorial description of the elementary differentials for $d=2$. 
We expect this result to be true for any $d\neq 1$; however, part of the proof cannot be generalised as it relies on a direct inspection with the help of a computer program.

In Section \ref{Sec::5}, we investigate the case when the number of $f$ is smaller than the number of nodes of $\tau$, and thus several nodes of $\tau$ are associated with the same non-linearity in the definition of the elementary differentials. 
We introduce in Definition \ref{l_labeling} the concept of $\ell$-labelling of a combinatorial object with $n$ nodes. 
This allows us to define the notion of an $\ell$-weak combinatorial description of the elementary differentials in dimension $d$ (see Definition \ref{l_weak_combinatorial}).

In Section \ref{Sec::6}, equipped with the definitions of the previous section, one shows in Theorems \ref{thm_1_b} and \ref{thm_2_b} the analogue of Theorems \ref{thm:1} and \ref{thm:2}, when one has two non-linearities for defining the elementary differentials, which already cover interesting cases for singular SPDEs (see Remark \ref{remark_SPDE}).
When one considers only a single non-linearity, the situation changes drastically. Indeed, Theorem~\ref{compatible_trees_b} shows that the compatibility with $\mathrm{RT}$ is no longer a natural condition to impose. Moreover, Theorem~\ref{compatible_multi_b} shows that one cannot hope for a result similar to Theorem~\ref{thm:2} in this context. The main issue is that it seems quite hard to get a precise description of the combinatorial sets compatible with rooted trees and multi-indices. One of the main reasons is that we cannot distinguish anymore the nodes of the combinatorial object due to the fact that one uses only a single non-linearity.

Finally, in the last section, we informally discuss the negative results we obtained. 
We argue that there is no ``meaningful'' combinatorial way to describe elementary differentials in fixed dimension aside from the rooted trees. 
Indeed, if we take the example of the dimension $2$, we can either get a $1$-weak combinatorial description that does not respect any algebraic structure of $\mathcal{W}_2$ or a description which is not compatible with the rooted tree nor faithful. None of those seems satisfactory.

\subsection*{Acknowledgements}

{\small
	Y. B. and P. L. gratefully acknowledge funding support from the European Research Council (ERC) through the ERC Starting Grant Low Regularity Dynamics via Decorated Trees (LoRDeT), grant agreement No.\ 101075208. Views and opinions expressed are however those of the author(s) only and do not necessarily reflect those of the European Union or the European Research Council. Neither the European Union nor the granting authority can be held responsible for them. 
	P. L. would also like to thank Thomas Agugliaro for the discussions about stable bases of representations while writing this article.
}

\section{Rooted trees, multi-indices and elementary differentials}
\label{Sec::2}

	A \emph{rooted tree} in a (non-empty) oriented finite graph $(V,E)$ with $V$ a finite set, the \emph{set of vertices}, and $E$ a set of (ordered) pairs of elements of $V$, the \emph{set of edges}, such that:
	\begin{itemize}
		\item Each vertex $v\in V$ admits a unique outgoing edge $e=(\bullet,v)\in E$, except one vertex $r\in V$ admitting no outgoing edges, which is called the \emph{root}.
		\item The graph $(V,E)$ is weakly connected; in this context it is equivalent to the fact that each vertex $v\in V$ admits a path to the root.
	\end{itemize}
A vertex $w$ is \emph{above} $v$ if there is a path from $w$ to $v$, we should notice that all the vertices are above the root. 
The set $C_v=\{w\in V\mid (v,w)\in E\}$ of vertices above $v$ and connected to $v$ by an edge are the \emph{children} of $v$.
A rooted tree is \emph{linear} if each vertex has at most one child.

A \emph{subtree} $\tau'=(V',E')$ of $\tau$ is a rooted tree that is a subgraph of $\tau$.
Let $A_v$ the set of vertices above $v$ (including $v$), we define $E_v$ as the subset of $E$ of edges between vertices of $A_v$; then $\tau_v=(A_v,E_v)$ is the \emph{subtree of $\tau$ above $v$}.

In the following, the set of vertices of a rooted tree will always be a subset of $\mathbb{N}$.
We denote by $\mathrm{RT}(n)$ the set of rooted trees with vertices $[n] = \{1,\dots,n\}$. The action of $\mathfrak{S}_n$ on $[n]$ naturally induces an action of $\mathfrak{S}_n$ on $\mathrm{RT}(n)$ by permuting the vertices.
Let $\tau=(V,E)$ a rooted tree, and $v_1,\dots,v_k$ the children of the root $r$. Then, one has
\begin{equs}
\tau = 	B^{r}_+(\tau_{v_1},\ldots,\tau_{v_k}) = 
\begin{tikzpicture}[scale=0.6,baseline=-2]
	\coordinate (root) at (0,-0.4);
	\coordinate (t1) at (-0.3,0.5);
	\coordinate (t2) at (-1.1,0.5);
	\coordinate (t3) at (1.1,0.5);
	\coordinate (t4) at (1.1,0.5);
	\coordinate (tau1) at (1,0.6);
	\draw[] (t1) -- (root);
	\draw[] (t2) -- (root);
	\draw[] (t3) -- (root);
	\node[not,label={[label distance=-0.2em]below: \scriptsize  $ r $}] (rootnode) at (root) {};
	\draw (-1.2,0.7) node[] {$\tau_{\tiny{v_{1}}}$};
	\node[not] (rootnode) at (root) {};
	\draw (1.3,0.7) node[] {$\tau_{\tiny{v_{k}}}$};
	\draw (-0.4,0.7) node[] {$\tau_{\tiny{v_{2}}}$};
	\draw (0.5,0.7) node[] {$\cdots$};
\end{tikzpicture}.
\end{equs}
Here, the operator $B_+^r$ connects the roots of the $\tau_{v_i}$ to a new root labelled by $r$. 
\begin{remark}\label{rmk:aromatictrees}
	We should notice that dropping the connectedness condition gives the definition of aromatic trees using the terminology of \cite{MV16} that we will denote $\mathrm{AR}(n)$. 
	Those combinatorial objects were first introduced in the context of B-series in \cite{IQT,CM07} with different terminology. 
	We refer to \cite{MV16} for their link with B-series with the aromatic B-series and to \cite{DL25} for an algebraic approach of aromatic B-series via the theory of combinatorial species. 
\end{remark}

In the next definition, we introduce the elementary differentials on $\mathbb{R}^d$.
\begin{definition} \label{elementary_diff}
	Let $d\in \mathbb{N}$, we denote $\mathcal{C}_d=\mathcal{C}^\infty(\mathbb{R}^d,\mathbb{R}^d)$, $f_j$ the $j$-coordinate of $f\in\mathcal{C}_d$, and $\partial_j=\frac{\partial}{\partial x_j}$. 
	Let us recursively define the map $F:\mathrm{RT}(n)\to\Hom(\mathcal{C}_d^{\otimes n},\mathcal{C}_d)$. 
	Let $\tau= B^{r}_+(\tau_{v_1},\ldots,\tau_{v_k})$, and $f^1,\dots,f^n\in \mathcal{C}_d$, and let
	\begin{multline*}
		F(\tau)\left((f^i)_{i\in \{1,\dots,n\}}\right) :=\\
		\sum_{j_1,\dots,j_k=1}^d F(\tau_{v_1})\left((f^i)_{i\in C_{v_1}}\right)_{j_1}\dots F(\tau_{v_k})\left((f^i)_{i\in C_{v_k}}\right)_{j_k}
		  \partial_{j_1}\dots\partial_{j_k}f^r
	\end{multline*}
	An \emph{elementary differential} of arrity $n$ of $\mathbb{R}^d$ is an element of $F(\mathrm{RT}(n))$. The \emph{vector space of elementary differentials} of arity $n$ of $\mathbb{R}^d$ is the vector space spanned by the elementary differentials of arity $n$.
	We denote by $\mathcal{W}_{d}(n)$ the vector space of elementary differentials of arity $n$ in $\mathbb{R}^d$, and by $\mathcal{LW}_d(n)$ the sub-vector space of $\mathcal{W}_d(n)$ of elementary differentials arising from linear trees.
\end{definition}

\begin{remark}
	The definition of elementary differentials can be extended to aromatic elementary differentials. We refer to \cite{MV16} for their actual definition.
\end{remark}

Let us compute the elementary differential associated to a tree on an example. Let $f,g,h,p\in\mathcal{C}_d$.
\begin{equs}
	F\left(
        \vcenter{\hbox{\begin{tikzpicture}
                %vertices
                \node (1) at (-0.4, 1.6) {};
                \node (2) at (-0.4, 0.8) {};
                \node (3) at (0, 0) {};
                \node (4) at (0.4, 0.8) {};
                %triangles
                %edges
                \draw[->,thick,shorten >= 2pt] (1)--(2);
                \draw[->,thick,shorten >= 2pt] (2)--(3);
                \draw[->,thick,shorten >= 2pt] (4)--(3);
                %dots
                %circles
                \draw[circle, fill=white] (1) circle [radius=6pt];
                \draw[circle, fill=white] (2) circle [radius=6pt];
                \draw[circle, fill=white] (3) circle [radius=6pt];
                \draw[circle, fill=white] (4) circle [radius=6pt];
                %labels
                \node at (1) {\scalebox{0.8}{$1$}};
                \node at (2) {\scalebox{0.8}{$2$}};
                \node at (3) {\scalebox{0.8}{$3$}};
                \node at (4) {\scalebox{0.8}{$4$}};
        \end{tikzpicture}}}
	\right)(f,g,h,p)=\sum_{i,j,k=0}^d f_i\partial_ig_jp_k\partial_j\partial_kh
\end{equs}

We should notice that for $d'\leq d$ we have a canonical projection $\mathcal{W}_d\twoheadrightarrow\mathcal{W}_{d'}$ by evaluating the $d-d'$ last variable to $0$ and forgetting the $d-d'$ last coordinates. 
Those projections are compatible with the map $F$. We will denote $\pi_d:\mathcal{W}_d\twoheadrightarrow\mathcal{W}_{1}$ the canonical projection from $\mathcal{W}_d$ to $\mathcal{W}_1$.

By definition, elementary differentials admit a combinatorial description with rooted trees; moreover, if the dimension is not fixed, this description is ``faithful'' thanks to the following classical result \cite[Sec. II.2, Exercise 4]{HNW}: 

\begin{proposition}
	Let $\tau_i \in \mathrm{RT}(n)  $ and $\lambda_i \in \mathbb{R}$. One has for $d$ large enough
	\begin{equs}
		\forall f^1,...,f^n \in \mathcal{C}_d, \quad \sum_{i} \lambda_i F(\tau_i)\left(f^1,...,f^n\right) = 0 \quad \Longleftrightarrow \quad \forall i, \quad \lambda_i = 0.
	\end{equs}
\end{proposition}

One can easily show that this result does not hold if one fixes the dimension. However, in dimension $1$, elementary differentials admit another combinatorial description.

\begin{definition}
	A \emph{multi-index} is a couple $(V,\varphi)$ with $V$ a finite set and $\varphi:V\to\mathbb{N}$ such that
	\begin{equs} \label{populated_condition}
		 \sum_{v\in V}\varphi(v)=|V|-1.
	\end{equs} 
	We denote by $\mathrm{MI}(n)$ the set of multi-indices on the set $[n] =\{1,\dots,n\}$.
\end{definition}
We may notice that any rooted tree $\tau=(V,E,r)$ gives a multi-index $(V,\varphi)$ with $\varphi(v)$ the number of children of $v$.
A multi-index is \emph{linear} if for each $v$, $\varphi(v)=0$ or $1$; this is exactly the multi-index of a linear tree. Let $\mathrm{LMI}(n)$ the set of linear multi-indices in $\mathrm{MI}(n)$.
One can use a polynomial representation of the multi-indices. Let $(z_{(i,k)})_{(i,k) \in [n] \times \mathbb{N}}$ some abstract variables. A multi-index can be denoted by
\begin{equs}
	z^{\beta} =  \prod_{(i,k) \in [n] \times \mathbb{N} } z_{(i,k)}^{\beta(i,k)}
\end{equs}
where $ \beta : [n] \times \mathbb{N} \rightarrow \lbrace 0,1 \rbrace $ such that $\beta(i,k)$ is non-zero for only one value of $k$. If $z^{\beta}$ is a linear multi-index, one has $i_0\in[n]$ such that
\begin{equs}
	z^{\beta} = z_{(i_0,0)} \prod_{ i \in [n]\setminus\{i_0\}  } z_{(i,1)}.
\end{equs}
With this new notation, the identity \eqref{populated_condition} becomes
\begin{equs}
	\sum_{(i,k) \in [n] \times \mathbb{N}} k \beta(i,k)   = n-1,
\end{equs}
where $ |\beta| = 	\sum_{(i,k) \in [n] \times \mathbb{N}}  \beta(i,k) = n $. One sets
\begin{equs}
	|\beta |_i :=	\sum_{k \in \mathbb{N}}  \beta(i,k).
\end{equs}
The surjective map $ \pi $ moving from rooted trees to multi-indices can also be described using the polynomial notation of multi-indices. For $ \tau $ a rooted tree with nodes set $V$, $\pi$ is given by
\begin{equs}
	\pi(\tau) = \prod_{v \in V} z_{(v,\phi(v))}.
\end{equs}
Let us compute $\pi$ on an example
\begin{equs}
	\pi\left(
	\vcenter{\hbox{\begin{tikzpicture}
                %vertices
                \node (1) at (-0.4, 1.6) {};
                \node (2) at (-0.4, 0.8) {};
                \node (3) at (0, 0) {};
                \node (4) at (0.4, 0.8) {};
                %triangles
                %edges
                \draw[->,thick,shorten >= 2pt] (1)--(2);
                \draw[->,thick,shorten >= 2pt] (2)--(3);
                \draw[->,thick,shorten >= 2pt] (4)--(3);
                %dots
                %circles
                \draw[circle, fill=white] (1) circle [radius=6pt];
                \draw[circle, fill=white] (2) circle [radius=6pt];
                \draw[circle, fill=white] (3) circle [radius=6pt];
                \draw[circle, fill=white] (4) circle [radius=6pt];
                %labels
                \node at (1) {\scalebox{0.8}{$1$}};
                \node at (2) {\scalebox{0.8}{$2$}};
                \node at (3) {\scalebox{0.8}{$3$}};
                \node at (4) {\scalebox{0.8}{$4$}};
        \end{tikzpicture}}}
	\right)=z_{(1,0)}z_{(2,1)}z_{(3,2)}z_{(4,0)}.
\end{equs}

One can define elementary differentials $G : \mathrm{MI}(n)\to\Hom(\mathcal{C}_1^{\otimes n},\mathcal{C}_1)$ associated with the multi-indices.
It is given for a multi-index $z^{\beta} \in \mathrm{MI}(n)$ and $ f_1,...,f_n \in \mathcal{C}_1 $ by
\begin{equs} \label{elementary_diff_multi}
	G(z^{\beta})\left(f_1,...,f_n\right) = f_1^{(|\beta |_1)} \cdots f_n^{(|\beta |_n)},
\end{equs}
where for $f \in \mathcal{C}_1$, we denote its $k$-derivative by  $f^{(k)}$ .
\begin{proposition} \label{prp:multi_trees}
	In dimension $1$, the map $F$ factorises through the multi-indices, meaning that we have for every $\tau \in \mathrm{RT}(n)$
	\begin{equs}
		F(\tau) = G(\pi(\tau)).
	\end{equs}
\end{proposition}
\begin{proof} 
	We proceed by induction on the size of the rooted tree. Let $\tau \in  \mathrm{RT}(n) $ with $\tau = B_{+}^r(\tau_{v_1},...,\tau_{v_k})$, we recall that $A_v$ is the set of vertices above $v$ in $\tau$. 
	Let $f_1,...,f_n \in \mathcal{C}_1$, one has
	\begin{equs}
		F(\tau)\left(f_1,...,f_n\right) & = F(\tau_{v_1})\left((f_i)_{i\in A_{v_1}}\right)\dots F(\tau_{v_k})\left(((f_i)_{i\in A_{v_k}}\right)
		f_r^{(k)}
	\end{equs}
	One notices that
	  $ k = \varphi(r) $
	   and by applying the induction hypothesis, one has
	\begin{equs}
		F(\tau_{v_i})\left((f_j)_{j\in A_{v_i}}\right)  &  = G(\pi(\tau_{v_i}))\left( (f_j)_{j\in A_{v_i}} \right)
		\\ &= G( \prod_{v \in A_{v_i}} z_{(v,\phi(v))})\left((f_j)_{j\in A_{v_i}}\right).
	\end{equs}
	By definition of $G$, it is easy to see that
	\begin{equs}
			F(\tau)\left(f_1,...,f_n\right) & = \prod_{i=1}^{k} G( \prod_{v \in A_{v_i}} z_{(v,\phi(v))})\left((f_j)_{j\in A_{v_i}}\right) f_r^{(k)}
			\\ &=  G( z_{(r,k)} \prod_{i=1}^{k} \prod_{v \in A_{v_i}} z_{(v,\phi(v))})\left(f_1,...,f_n\right)
			\\ & = G(\pi(\tau))\left(f_1,...,f_n\right)
	\end{equs}
	which concludes the proof.
\end{proof}

Let us illustrate this proposition with an example, let $f,g,h,p\in\mathcal{C}_1$
\begin{equs}
        F\left(
        \vcenter{\hbox{\begin{tikzpicture}
                %vertices
                \node (1) at (-0.4, 1.6) {};
                \node (2) at (-0.4, 0.8) {};
                \node (3) at (0, 0) {};
                \node (4) at (0.4, 0.8) {};
                %triangles
                %edges
                \draw[->,thick,shorten >= 2pt] (1)--(2);
                \draw[->,thick,shorten >= 2pt] (2)--(3);
                \draw[->,thick,shorten >= 2pt] (4)--(3);
                %dots
                %circles
                \draw[circle, fill=white] (1) circle [radius=6pt];
                \draw[circle, fill=white] (2) circle [radius=6pt];
                \draw[circle, fill=white] (3) circle [radius=6pt];
                \draw[circle, fill=white] (4) circle [radius=6pt];
                %labels
                \node at (1) {\scalebox{0.8}{$1$}};
                \node at (2) {\scalebox{0.8}{$2$}};
                \node at (3) {\scalebox{0.8}{$3$}};
                \node at (4) {\scalebox{0.8}{$4$}};
        \end{tikzpicture}}}
	\right)(f,g,h,p)=G\left(z_{(1,0)}z_{(2,1)}z_{(3,2)}z_{(4,0)}\right)(f,g,h,p)=fg'ph''.
\end{equs}

Once again, this combinatorial description of elementary differential is faithful thanks to the following classical result:

\begin{proposition}
		Let $m_i \in \mathrm{MI}(n)  $ and $\lambda_i \in \mathbb{R}$. One has 
\begin{equs}
	\forall f_1,...,f_n \in \mathcal{C}_1, \quad \sum_{i} \lambda_i G(m_i)\left(f_1,...,f_n\right) = 0 \quad \Longleftrightarrow \quad \forall i, \quad \lambda_i = 0.
\end{equs}
\end{proposition}
\begin{proof}
	For $m=([n],\varphi)\in\mathrm{MI}(n)$, let $(f^m_i)_{i\in\{1,\dots,n\} }$ defined by 
	$$f^m_i=\frac{x^{\varphi(i)}}{(\varphi(i))!}.$$ 
	Let $m'\in\mathrm{MI}(n)$. It is clear that $G(m')(f^m_1,\dots,f^m_n)=1$ if $m=m'$, and $0$ if $m'\neq m$. 
	\end{proof}

	We already mentioned ``combinatorial descriptions'' of elementary differentials. However, since those are the objects we would like to study, let us give a formal definition. 	First, one can notice that since elements of $\mathcal{W}_d(n)$ are $n$-linear maps, $\mathcal{W}_d(n)$ is naturally a representation of $\mathfrak{S}_n$ by permutation of the $n$-inputs. Then, one has
	
	\begin{definition} \label{Def_combinatorial}
		A \emph{combinatorial description} of $\mathcal{W}_d$ is a couple $(\mathrm{CD}_d,H_d)$ with $(\mathrm{CD}_d(n))_{n\in\mathbb{N}}$ a sequence of sets such that $\mathrm{CD}_d(n)$ is endowed with a action of $\mathfrak{S}_n$, and $(H_d^n)_{n\in\mathbb{N}}$ a sequence of maps $H_d^n:\mathrm{CD}_d(n)\to\mathcal{W}_d(n)$ such that $H_d^n$ is $\mathfrak{S}_n$-equivariant, and the set $H_d^n(\mathrm{CD}_d(n))$ generates $\mathcal{W}_d$.
		A combinatorial description is \emph{faithful} if for all $n\in\mathbb{N}$, the set $H_d^n(\mathrm{CD}_d(n))$ is a basis of $\mathcal{W}_d(n)$.
	A combinatorial description is \emph{compatible with $\mathrm{RT}$} if we have a sequence of equivariant surjective maps $p_n:\mathrm{RT}(n)\to\mathrm{CD}_d(n)$ such that for every $ \tau \in \mathrm{RT}(n) $
	\begin{equs}
		F(\tau) = H^n_d( p_n(\tau) ).
	\end{equs} 
		A combinatorial description is \emph{compatible with $\mathrm{MI}$} if we have a sequence of equivarient surjective maps $q_n:\mathrm{CD}_d(n)\to\mathrm{MI}(n)$ such that for $a \in \mathrm{CD}_d(n) $
		\begin{equs}
			\pi_d(H_d^n( a )) = G( q_n(a) ).
		\end{equs}
		where $\pi_d$ is the canonical projection $\mathcal{W}_d\to\mathcal{W}_1$.
	\end{definition}
	
	By definition, the rooted trees provides a combinatorial description of the elementary differential in all dimensions. Moreover, the multi-indices are a faithful combinatorial description of the elementary differentials in dimension $1$. It would be very interesting to get combinatorial descriptions in dimension $d\neq 1$, and even more if those are faithful or compatible with $\mathrm{RT}$ or $\mathrm{MI}$.
	In the next sections, we discuss these points and provide negative results
	such as Theorems \ref{thm:1} and \ref{thm:2}.

	We should point out that those results do not state that for $d\neq 1$, $\mathcal{W}_d$ does not admit any combinatorial description (which is false since $\mathrm{RT}$ is always a combinatorial description) nor any \textit{faithful} combinatorial description. 
	However, such a faithful combinatorial description cannot be compatible with the rooted trees nor even the multi-indices, and thus should be quite involved, and possibly unusable in practice, if it even exists.
	
\section{Combinatorial description compatible with rooted trees}
\label{Sec::3}
	
	The main theorem of this section is the following
	\begin{theorem} \label{thm:1}
		In dimension $d\neq 1$, $\mathrm{RT}$ is the only combinatorial description of $\mathcal{W}_d$ compatible with $\mathrm{RT}$.
	\end{theorem}
	 Since any linear relation of the elementary differentials that holds in dimension $d$ also holds in dimension $d'\leq d$, it is enough to work in dimension $d=2$. 
	 In the sequel, we will use the following notation for $f \in \mathcal{C}_2$ and $(x,y) \in \mathbb{R}^2$: 
	 \begin{equs}
	 	f(x,y) = f_1(x,y) \vec{e}_x + f_2(x,y) \vec{e}_y
	 \end{equs}
	 where $ \vec{e}_x  = (1,0) $ and $ \vec{e}_y = (0,1) $.
	 Before giving the proof, we need a couple of definitions and propositions.
	
	\begin{definition} \label{def_sigma}
		Let $\tau=(V,E)$ a rooted tree with $r$ its root, and $\sigma$ a bijection of $V$. We define $E_\sigma=\{(v,\sigma(w))\mid (v,w)\in E\}$. 
		Then the oriented graph $\sigma.\tau:=(V,E_\sigma)$ is an oriented graph such that each vertex admits a unique outgoing edge, except $\sigma(r)$, which admits none. 
		However, $(V,E_\sigma)$ is not necessarily weakly connected (by Remark~\ref{rmk:aromatictrees}, $(V,E_\sigma)$ is an aromatic tree but not necessarily a rooted tree). 
		We will always consider bijection $\sigma$ such that $(V,E_\sigma)$ is weakly connected and thus is a rooted tree.
	\end{definition}

	\begin{remark}
		We should notice that if $\tau$ is an aromatic tree, then $\sigma.\tau$ is also an aromatic tree. 
		In fact, aromatic trees are the correct context for the above definition since it provides us with an action of $\mathfrak{S}_n$ on $\mathrm{AT}(n)$ which is not the usual action by permutation of the vertices.
	\end{remark}
	
	Let us compute $\sigma.\tau$ on an example with $\sigma=(1\;3\;2)$
	\begin{equs}
		(1\;3\;2).\left(\vcenter
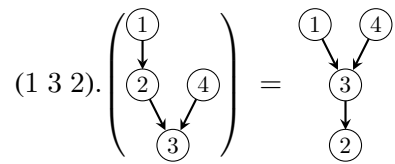
{\hbox{\begin{tikzpicture}
                %vertices
                \node (1) at (-0.4, 1.6) {};
                \node (2) at (-0.4, 0.8) {};
                \node (3) at (0, 0) {};
                \node (4) at (0.4, 0.8) {};
                %triangles
                %edges
                \draw[->,thick,shorten >= 2pt] (1)--(2);
                \draw[->,thick,shorten >= 2pt] (2)--(3);
                \draw[->,thick,shorten >= 2pt] (4)--(3);
                %dots
                %circles
                \draw[circle, fill=white] (1) circle [radius=6pt];
                \draw[circle, fill=white] (2) circle [radius=6pt];
                \draw[circle, fill=white] (3) circle [radius=6pt];
                \draw[circle, fill=white] (4) circle [radius=6pt];
                %labels
                \node at (1) {\scalebox{0.8}{$1$}};
                \node at (2) {\scalebox{0.8}{$2$}};
                \node at (3) {\scalebox{0.8}{$3$}};
                \node at (4) {\scalebox{0.8}{$4$}};
		\end{tikzpicture}}}\right)\;=\;
                \vcenter{\hbox{\begin{tikzpicture}
                %vertices
                \node (1) at (-0.4, 1.6) {};
                \node (2) at (0, 0) {};
                \node (3) at (0, 0.8) {};
                \node (4) at (0.4, 1.6) {};
                %triangles
                %edges
                \draw[->,thick,shorten >= 2pt] (1)--(3);
                \draw[->,thick,shorten >= 2pt] (3)--(2);
                \draw[->,thick,shorten >= 2pt] (4)--(3);
                %dots
                %circles
                \draw[circle, fill=white] (1) circle [radius=6pt];
                \draw[circle, fill=white] (2) circle [radius=6pt];
                \draw[circle, fill=white] (3) circle [radius=6pt];
                \draw[circle, fill=white] (4) circle [radius=6pt];
                %labels
                \node at (1) {\scalebox{0.8}{$1$}};
                \node at (2) {\scalebox{0.8}{$2$}};
                \node at (3) {\scalebox{0.8}{$3$}};
                \node at (4) {\scalebox{0.8}{$4$}};
		\end{tikzpicture}}}
	\end{equs}
	This computation can be done inductively.
	Since the root of $\tau$ is $3$, and $\sigma(3)=2$, the root of $\sigma.\tau$ is $2$. 
	In $\tau$, the child of $2$ is $1$, thus in $\sigma.\tau$, the child of $2$ is $\sigma(1)$ which is $3$. 
	Finally, the children of $3$ in $\tau$ are $2$ and $4$, thus the children of $3$ in $\sigma.\tau$ are $\sigma(2)$ and $\sigma(4)$, which are $1$ and $4$.

	The previous definition does not provide the canonical action of the permutation on the rooted trees; however, we may notice that $\tau$ and $\sigma.\tau$ have the same multi-index.
	
	\begin{proposition} \label{prp:geofixed}
		Let $\tau=(V,E)$ a rooted tree, and $\sigma$ such that $\sigma.\tau$ is a rooted tree and $F(\tau)=F(\sigma.\tau)$. 
		Then for any $v\in V$, we have that $\sigma(C_v)=C_v$. Thus one has $E_\sigma=E$, and
		\begin{equs}
			F(\sigma . \tau) = F(\tau)  \quad \Longleftrightarrow \quad \sigma . \tau = \tau.
		\end{equs}
	\end{proposition}
	
	\begin{proof}
		Let $v\in V$ a vertex of $\tau$, and $C_v=\{v_1,\dots,v_k\}$.
		One sets 
		\begin{equs}
		        f^v= f^{v_i}= y^k\vec{e}_x + x^n\vec{e}_y, \, \, \, \forall i\in [k],
			\quad 
			f^w=x^n\vec{e}_x, \, \, \, \forall w\in V\setminus\{v,v_1,\dots,v_k\}.
		\end{equs}
		It is clear that $F(\tau)((f^v)_{v\in V})=g(x)\vec{e}_x+h(x,y)\vec{e}_y$ with $g\neq 0$ for $n$ large enough. Moreover, $h\neq 0$ if and only if $v$ is the root.  
		This is due to the $k$ derivatives $ \partial_y $ coming from the $f^{v_i}$ that act on $f^v$, thus no $y$ from $f^v$ will end up in $g$. 
		Moreover, no $y$ of $f^{v_i}$ will end up in $g$, since the $\vec{e}_x$ component of $f^{v_i}$ will be multiplied with the $\vec{e}_y$ component of $f^v$, and either vanish if $v$ is not the root, or end up in $h(x,y)\vec{e}_y$ if $v$ is the root.

		If $\sigma(v_i)\notin C_v$, then the $\vec{e}_x$ component of $F(\sigma.\tau)((f^w)_{w\in V})$ will depend on $y$. Indeed, the only way to get a function that does not depend on $y$ is to have the $k+1$ functions $y^k\vec{e}_x + x^n\vec{e}_y$ forming a corolla, and thus that $\sigma(C_v)=C_v$. 

		Thus if $F(\tau)=F(\sigma.\tau)$, for each vertex $v$ we have $(v,w)\in E\Leftrightarrow (v,w)\in E_\sigma$, thus $E=E_\sigma$, and $\tau=\sigma.\tau$.
	\end{proof}

	\begin{remark}
		We should notice that this proposition is still true if one assumes that $\tau=(V,E)$, and $\sigma.\tau$ are aromatic trees. Indeed, if $(v,v)\notin E$ then the same proof holds. If $(v,v)\in E$, then setting
                \begin{equs}
                        f^v= x^ny\vec{e}_y , 
                        \quad 
                        f^w=x^n\vec{e}_x, \, \, \, \forall w\in V\setminus\{v\},
                \end{equs}
		ensures that $\sigma(v)=v$. Let us choose $v'\in C_v$, setting
                \begin{equs}
			f^v= f^{v'}= x^ny^2\vec{e}_y,
                        \quad
                        f^w=x^n\vec{e}_x, \, \, \, \forall w\in V\setminus\{v,v'\},
                \end{equs}
		ensures that $\sigma(v')\in C_v$.
	\end{remark}
	
	\begin{proposition} \label{prp:transitiveaction}
		Let $\tau_1=(V,E_1),\tau_2=(V,E_2)$ two rooted trees with the same multi-index (and thus the same set of vertices). Then, there exists a bijection $\sigma$ of $V$ such that $\sigma.\tau_1=\tau_2$.
	\end{proposition}
	
	\begin{proof}
		Let us construct $\sigma$. Let $r_1,r_2$ the roots of $\tau_1,\tau_2$, we set $\sigma(r_1)=r_2$,
			 For any $v\in V$, let $\{v_1,\dots,v_k\}$ its children in $\tau_1$, and $\{w_1,\dots,w_k\}$ its children in $\tau_2$, and we set $\sigma(v_i)=w_i$.
		
		Since any vertex is either the root or the child of another vertex, $\sigma$ is well-defined. We have, in fact, several choices for $\sigma$ depending on the order we choose for the children of $v$ for each vertex $v\in V$.
		It is clear from the construction that $\sigma.\tau_1=\tau_2$.
	\end{proof}
	
	\begin{remark}
		We should notice that the previous proposition still holds if one assumes that $\tau_1,\tau_2$ are aromatic trees. Indeed, one can iteratively construct $\sigma$ on each connected component.
	\end{remark}
	
	\begin{proof}[of Theorem~\ref{thm:1}]
		Let $\tau_1,\tau_2$ such that $\lambda F(\tau_1)=F(\tau_2)$ with $\lambda\in\mathbb{R}$ in dimension $2$. 
		The same identity is also true in dimension $1$. Indeed, one can take in dimension $2$, for every $ v  \in V_1 $
		\begin{equs}
			\bar{f}^v = f^v(x) \vec{e}_x.
		\end{equs}
		which will give us the identity in dimension $1$. From Proposition \ref{prp:multi_trees},
		one has $\lambda G(m_1)=G(m_2)$ in dimension $1$, where $m_1, m_2$ are the multi-indices associated with $\tau_1,\tau_2$, thus $m_1=m_2$, and $\lambda=1$. 
		Hence, from Proposition \ref{prp:transitiveaction}, there exists $\sigma$ such that $\sigma.\tau_1=\tau_2$.
		Then
		\begin{equs}
			F(\sigma . \tau_1)  = F(\tau_2) = F(\tau_1).
		\end{equs} 
		Therefore from Proposition \ref{prp:geofixed}, one has
		\begin{equs}
			\sigma . \tau_1 = \tau_1
			\end{equs}
		which allows us to say that $\tau_1=\tau_2$ and this concludes the proof.	
	\end{proof}

	\begin{remark}
		Since Propositions~\ref{prp:geofixed} and~\ref{prp:transitiveaction} also hold for aromatic trees, Theorem~\ref{thm:1} holds for aromatic trees.
	\end{remark}
	
\section{Combinatorial description compatible with multi-indices}
	
	\label{Sec::4}
	
	In this section, we investigate  the existence of combinatorial sets compatible with the multi-indices, which is a weaker requirement than compatibility with the rooted trees. Our main negative result is the following
	
	\begin{theorem} \label{thm:2}
		In dimension $d\neq 1$, $\mathcal{W}_d$ does not admit any faithful combinatorial description compatible with $\mathrm{MI}$.
	\end{theorem}

	The proof of the previous theorem is based on the framework of \cite[Thm 3.3]{D00}, which implies our statement. We start by recalling the definition of the Witt algebras.
	
	\begin{definition} \label{def_Witt}
		The \emph{Witt algebra of rank $d$} is the algebra $(W_d,\trigl)$ with $W_d$ the vector space of polynomial maps $\mathbb{R}^d\to\mathbb{R}^d$, and $\trigl$ defined by:
		$$f\trigl g=\sum_{i=1}^d g_i\partial_i f$$
		with $f_i$ the $i$-component of $f$, and $\partial_i=\frac{\partial}{\partial x_i}$. 
	\end{definition}
		
	This algebra is a pre-Lie algebra (also known as right-symmetric algebras), meaning that it satisfies:
	$$f\trigl (g\trigl h) - (f\trigl g)\trigl h = f\trigl (h\trigl g) - (f\trigl h)\trigl g$$
	for any $f,g,h\in W_d$. In particular, $W_d$ is not an associative algebra.
	The above relation is an ``identity'' of $W_d$. Formally, an \emph{identity} of $W_d$ is a linear combination of formal compositions of $\trigl$ such that the induced multilinear map $W_d^{\otimes n}\to W_d$ is $0$. 
	A \emph{left identity} is an identity where all the formal compositions are done in the right input of $\trigl$. For example, $f\trigl (g\trigl h)-g\trigl (f\trigl h)$ is an identity of $W_1$, moreover, this is a left identity. 
	We should notice that this is not an identity of $W_d$ for $d\neq 1$. We recall \cite[Thm 3.3]{D00} in the next Theorem below:

	\begin{theorem} \label{minimal_identity}
		Let $\mathrm{s}_{2d}$ be the following left identity:
		$$\mathrm{s_{2d}}:=\sum_{\sigma\in\mathfrak{S}_{2d}}\varepsilon(\sigma) f_{\sigma(1)}\trigl(\dots\trigl(f_{\sigma(2d)}\trigl f_{2d+1})\dots)$$
		The vector space spanned by $\mathrm{s}_{2d}$ and its permutations are the only left identities of $W_d$ in arity $2d+1$. 
		Moreover, $W_d$ does not satisfy any left identities in arity $2d$ or less.
	\end{theorem}
	
	\begin{remark}\label{rmk:key}
		A direct consequence of this theorem is that the representation of $\mathfrak{S}_{2d+1}$ of left identities of $W_d$ of arity $2d+1$ is the product of the signature representation with the natural permutation representation.
	\end{remark}

	To link this theorem with our questions, we need the following results:
	
	\begin{proposition} \label{correspondance_left}
		One has a one-to-one correspondence between the left identities of $W_d$ and the linear relation satisfied by the elementary differentials arising from linear trees.
	\end{proposition}
	
	\begin{proof}
		The product $\trigl$ can easily be defined on $\mathcal{C}_d$, moreover, by density of the polynomial, the identities in $W_d$ are the same as the identities in $\mathcal{C}_d$.
		To conclude, one needs to check that any elementary differential arising from a linear tree is in bijection with the chains of right compositions of $\trigl$.
		For that, we need to recall the grafting product $\curvearrowright$ defined on rooted trees by
		\begin{equs}
			\sigma \curvearrowright \tau = \sum_{v \in V} \sigma \curvearrowright_v \tau
		\end{equs}
		where $ V $ are the nodes of $\tau$ and $ \sigma \curvearrowright_v \tau $ creates a new rooted tree by connecting the root of $ \sigma $ with the node $v$ of $\tau$. We denote by $V'$ the nodes of $\sigma$.
		One has 
		\begin{equs} \label{morphism_property}
			F(\sigma  \curvearrowright \tau)((f^w)_{w \in V'\cup V}) = 	F(\tau )((f^v)_{v \in V}) \trigl
			F(\sigma)( (f^u)_{u \in V'}).
		\end{equs}
		Any linear tree $ \tau $ can be decomposed into
		\begin{equs}
			\tau = ((\bullet_{v_1} \curvearrowright \bullet_{v_2}) \curvearrowright  \cdots) \curvearrowright \bullet_{r}
		\end{equs}
			where the $v_i$ are the nodes of $\tau$ different from the root $r$. Applying the elementary differential to the previous expression and using the morphism property \eqref{morphism_property}, one is able to conclude.
	\end{proof}
	
	\begin{remark}
		In fact, as a consequence of the main result of \cite{E94}, now well known due to \cite{CL01}, which states that the labelled rooted trees endowed the grafting product $\curvearrowright$ are the free pre-Lie algebra, together with \eqref{morphism_property}, any elementary differential can be obtained as a linear combination of compositions of $\trigl$.
	\end{remark}
	
	It only remains to state and prove the corollary of \cite[Thm 3.3]{D00} allowing us to get the second negative result we stated.
	
	\begin{corollary} \label{representation_d}
		Let us recall that $\mathcal{LW}$ denotes the space of elementary differentials arising from linear trees.
		For $d\neq 1$, the representation $\mathcal{LW}_d(2d+1)$ does not admit a basis stable by action of $\mathfrak{S}_{2d+1}$.
	\end{corollary}
	
	\begin{proof}
		To do so, let us compute $\rchi$ the character of $\mathcal{LW}_d(2d+1)$. 
		We recall that for $\rho:G\to\mathrm{Gl}_n(V)$ a representation of $G$, its character is $\rchi_\rho:g\mapsto \mathrm{tr}(\rho(g))$ from $G$ to $\mathbb{R}$. 
		Moreover, we recall that $\rchi_\rho(hgh^{-1})=\rchi_\rho(g)$, and that if $V$ admits a basis $B$ stable by the action of $G$ given by $\rho$, then $\rchi_\rho(g)$ is the number of elements of $B$ fixed by $g$.

		We know that $\mathcal{LW}_d(2d+1)$ is the quotient of the regular representation of $\mathfrak{S}_{2d+1}$ by the representation given by the left identities of $W_d$ of arity $2d+1$. 
		Thus by Remark~\ref{rmk:key}, we can compute $\rchi$.
		Let $\sigma=(1\;\dots\;2d-1)$ a $(2d-1)$-cycle, then $\rchi(\sigma)=0-\varepsilon(\sigma)\times 2=-2$.

		So if $\mathcal{LW}_d(2d+1)$ had a basis stable by the action of $\mathfrak{S}_{2d+1}$, then $\sigma$ would fixes $-2$ elements of this basis, which is absurd.
	\end{proof}

	\begin{proof}[of Theorem~\ref{thm:2}]
		A faithful combinatorial description of $\mathcal{W}_d$ is exactly a sequence of basis of $\mathcal{W}_d(n)$ stable by the action of $\mathfrak{S}_n$. 
		Let $\mathrm{CD}_d$ a combinatorial description of $\mathcal{W}_d$ compatible with the multi-indices via the sequence of surjections $q_n : \mathrm{CD}_d(n) \rightarrow \mathrm{MI}(n) $.
		 Let $\mathrm{LCD}_d(n)=q_n^{-1}(\mathrm{LMI}(n))$ the preimage of the linear multi-indices. Then if $\mathrm{CD}_d$ is faithful, $\mathrm{LCD}_d(n)$ is a basis of $\mathcal{LW}_d(n)$ stable by the action of $\mathfrak{S}_n$.
		By the previous corollary, this is not possible if $d\neq 1$ since $\mathcal{LW}_d(2d+1)$ does not admit such a basis.
	\end{proof}

	In dimension $d=2$ it is possible to strengthen this result:
	
	\begin{theorem}\label{thm:3}
		There is no faithful combinatorial description of $\mathcal{W}_2$.
	\end{theorem}

	To show this result, we need to check the following fact:

	\begin{proposition}\label{prp:comp}
		In dimension $2$, the only identities of arity $5$ are the linear combinations of permutations of the identity $\mathrm{s}_4$ of \cite[Thm 3.3]{D00}.
	\end{proposition}

	\begin{proof}
		We conjecture this fact to be true in any dimension; the current proof of this result in dimension $2$ is done via brute force with a computer by computing the dimension of $\mathcal{W}_2(5)$ which is $620$, while the cardinality of $\mathrm{RT}(5)$ is $625$. Since $\mathrm{s}_4$ spans a representation of dimension $5$, those are the only identities in arity $5$ and dimension $2$.
	\end{proof}

	\begin{proof}[of Theorem~\ref{thm:3}]
		Let us compute the character $\rchi_2$ of the representation $\mathcal{W}_2(5)$.
		Let $V$ be the representation spanned by $\mathrm{s}_4$, we denote $\rchi_V$ its character, and let $\mathcal{W}_\infty(5)$ the representation spanned by $\mathrm{RT}(5)$, we denote $\rchi_\infty$ its character.
		By the previous proposition, we have that $\rchi_2=\rchi_\infty-\rchi_V$; moreover, by Remark~\ref{rmk:key}, we know the representation $V$. 

		To compute the character of a representation $R$, we recall that if $R$ admits a basis $B$ stable by the action of $G$, then $\rchi_R(g)$ is the number of elements of $B$ fixed by $g$.
		Thus we get the following table:
		\begin{equs}
			\begin{array}{r|c|c|c|c|c|c|c}
				& \mathrm{id} & (1\;2) & (1\;2\;3) & (1\;2)(3\;4) & (1\;2\;3\;4) & (1\;2)(3\;4\;5) & (1\;2\;3\;4\;5)\\
				\hline
				\rchi_\infty & 625 & 27 & 4 & 3 & 1 & 0 & 0 \\
				\hline
				\rchi_V & 5 & -3 & 2 & 1 & -1 & 0 & 0 \\
				\hline
				\rchi_2 & 620 & 30 & 2 & 2 & 2 & 0 & 0
			\end{array}
		\end{equs}
		Assume that $\mathcal{W}_2(5)$ admits a basis $B$ stable by the action of $\mathfrak{S}_5$, let $X$ be the orbit of one of the elements of $B$ fixed by $(1\;2\;3\;4)$. 
		Let us find $\rchi_X$. We know that $\rchi_X(\sigma)\leq\rchi_2(\sigma)$ for any $\sigma\in\mathfrak{S}_5$. 
		Moreover, since $(1\;2\;3\;4)^2=(1\;3)(2\;4)$ which is conjugated to $(1\;2)(3\;4)$, we have $\rchi_X((1\;2)(3\;4))=\rchi_X((1\;2\;3\;4))\geq 1$.
		Finally, since the action on $X$ is transitive, $X=\mathfrak{S}_5/G$ with $G$ a subgroup of $\mathfrak{S}_5$. 
		Let us sum up those constraints in this table, with $a\geq 1$, and the symbol $*$ when we do not impose any constraint.
		\begin{equs}
                        \begin{array}{r|c|c|c|c|c|c|c}
                                & \mathrm{id} & (1\;2) & (1\;2\;3) & (1\;2)(3\;4) & (1\;2\;3\;4) & (1\;2)(3\;4\;5) & (1\;2\;3\;4\;5)\\
                                \hline
                                \rchi_X & * & * & * & a & a & 0 & 0 \\
                                \hline
				\rchi_{\mathfrak{S}_4} & 5 & 3 & 2 & 1 & 1 & 0 & 0 \\
                        \end{array}
                \end{equs}
		Among the $19$ conjugacy classes of subgroups of $\mathfrak{S}_5$, only one, $\mathfrak{S}_4$ with the natural inclusion, satisfies the constraints given in the above table. However, since $\rchi_{\mathfrak{S}_4}((1\;2\;3))=2$, we get 
		$$4=2\times\rchi_{\mathfrak{S}_4}((1\;2\;3))\leq\rchi_2((1\;2\;3))=2$$
		which is a contradiction.
	\end{proof}

\section{Labelling combinatorial objects}
\label{Sec::5}

	The result obtained so far may seem unsatisfactory since in the definition of a combinatorial description, we allow the evaluation of elementary differentials on an arbitrary number of distinct functions, while in the case of B-series, the elementary differentials are only evaluated on copies of the same function. Even when using splitting methods, the number of distinct functions on which the elementary differentials are evaluated is bounded and usually quite small. Let us bound the number of distinct functions on which we evaluate the elementary differential. To do so, let us label our combinatorial object.

	Let $\ell\in\mathbb{N}$, we recall that we denote $[\ell]=\{1,\dots,\ell\}$. 
	Moreover, we have an action of $\mathfrak{S}_n$ on $[\ell]^n$ given by permutation of the coefficients. 
	Let $k_1,\dots,k_\ell\in\mathbb{N}$ such that $\sum_{i=1}^\ell k_i=n$, we define $[\ell]^{k_1,\dots,k_\ell}$ as the subset of $[\ell]^n$ such that $x\in[\ell]^{k_1,\dots,k_\ell}$ if $x$ contains $k_i$ occurrences of $i$. We should notice that all the $[\ell]^{k_1,\dots,k_\ell}$ are stable by the action of $\mathfrak{S}_n$, and that they even provide the partition of $[\ell]^n$ in the orbits of the action of $\mathfrak{S}_n$.

	\begin{definition} \label{l_labeling}
		Let $\mathrm{C}=(\mathrm{C}(n))_{n\in\mathbb{N}}$ endowed with an action of $\mathfrak{S}_n$ on $\mathrm{C}(n)$. The \emph{$\ell$-labelling} of $\mathrm{C}$ is the set $\mathrm{L}(\mathrm{C})$ defined by:
		$$\mathrm{L}(\mathrm{C})=\bigcup_{n\in\mathbb{N}}\mathrm{C}(n)\times_{\mathfrak{S}_n}[\ell]^n$$
		where $\times_{\mathfrak{S}_n}$ is the product of $\mathfrak{S}_n$-set, meaning that for $G$ a group and $A$, $B$ two sets with an action of $G$, then: 
		$$A\times_G B = (A\times B)/((a,g.b)\sim (g.a,b)).$$
		We should notice that $\mathrm{L}(\mathrm{C})$ splits according to the $[\ell]^{k_1,\dots,k_\ell}$. More precisely, for $k_1,\dots,k_\ell\in\mathbb{N}$ such that $\sum_{i=1}^\ell k_i=n$, we define:
		$$\mathrm{L}^{k_1,\dots,k_\ell}(\mathrm{C})=\mathrm{C}_n\times_{\mathfrak{S}_n}[\ell]^{k_1,\dots,k_\ell}.$$
		Then we have:
		$$\mathrm{L}(\mathrm{C})=\sum_{k_1,\dots,k_\ell\in\mathbb{N}}\mathrm{L}^{k_1,\dots,k_\ell}(\mathrm{C}).$$
	\end{definition}

	We should explain why this is the formalisation of a ``labelling'' of $\mathrm{C}$. Indeed, $\mathrm{C}(n)$ should be understood as the set of objects with $n$ ``vertices'' of $\mathrm{C}$ such that the vertices are numbered from $1$ to $n$ to be able to distinguish them. 
	Moreover, the action of $\mathfrak{S}_n$ is exactly the permutation of those vertices. 
	Then a ``labelling'' of $\mathrm{C}$ should correspond to giving a label to each vertex; thus, we can encode this by an element $\underline{c}_x=(c,(x_1,\dots,x_n))\in\mathrm{C}(n)\times[\ell]^n$, where $c$ is the element of $\mathrm{C}(n)$, and $x_i$ is the ``label'' of the $i$-th ``vertex'' of $c$. 
	However, permuting some ``vertices'' should be the same as permuting their labels, which is why one needs to quotient by the diagonal action. 
	Thus, an element of $\mathrm{C}$ ``labelled'' over $[\ell]$ is an element of $\mathrm{C}(n)\times_{\mathfrak{S}_n}[\ell]^n$ for some $n$, hopefully explaining the definition.
This kind of description of a labelling is not new at all. This is the way labelling is managed in the theory of combinatorial species.

	\begin{definition}
		Let $\varphi:[\ell]\to[\ell]$, for $k_1,\dots,k_\ell\in\mathbb{N}$ we define $k^\varphi_1,\dots,k^\varphi_\ell$ by:
		$$k^\varphi_i=\sum_{j\in \varphi^{-1}(i)}k_j$$
		we should notice that $(k^\varphi)^\psi_i=k^{\psi\circ \varphi}_i$. The map $\varphi$ induces a map $\varphi_*$:
		$$\begin{array}{lcccc}
			\varphi_*&:&\mathrm{L}^{k_1,\dots,k_\ell}(\mathrm{C})&\to&\mathrm{L}^{k^\varphi_1,\dots,k^\varphi_\ell}(\mathrm{C})\\
			& & (c,(x_1,\dots,x_n)) & \mapsto & (c,(\varphi(x_1),\dots,\varphi(x_n)))
		\end{array}$$
		Moreover, we should notice that $\varphi_*\circ \psi_*=(\varphi\circ \psi)_*$. The maps $\varphi_*$ will be called \emph{identification maps}, since they correspond to identifying and permuting the labels.
	\end{definition}

	We can now define the labelled elementary differentials, allowing us to constrain some functions to be the same when evaluating the elementary differential.
	\begin{definition}
		Let $\ell\in \mathbb{N}$ and $k_1,\dots,k_\ell\in\mathbb{N}$. Let us denote by $\Hom^{k_1,\dots,k_\ell}(\mathcal{C}_d^{\times\ell},\mathcal{C}_d)$ the set of functions $\mathcal{C}^{\times\ell}_d\to\mathcal{C}_d$ that are $k_i$ homogeneous in the $i$-th input. We define 
		$$F^{k_1,\dots,k_\ell}:\mathrm{L}^{k_1,\dots,k_\ell}(\mathrm{RT})\to\Hom^{k_1,\dots,k_\ell}(\mathcal{C}_d^{\times\ell},\mathcal{C}_d)$$
		by \begin{equs}
		F^{k_1,\dots,k_\ell}((\tau,(x_1,\dots,x_n))(f^1,\dots,f^\ell)=F(\tau)((f^{x_i})_{i\in \{1,\dots,n \}}).
		\end{equs}
		We denote by $\mathcal{W}_d^{k_1,\dots,k_\ell}$ the vector space spanned by $F^{k_1,\dots,k_\ell}(\mathrm{L}^{k_1,\dots,k_\ell}(\mathrm{RT}))$. 
		We should notice that the identification maps are  defined on $\Hom^{k_1,\dots,k_\ell}(\mathcal{C}_d^{\times\ell},\mathcal{C}_d)$, and thus that they are well-defined on $\mathcal{W}^{k_1,\dots,k_\ell}_d$. 
		Moreover, let $\varphi_*$ an identification map, then the following diagram commutes:
		\[\begin{tikzcd}
			\mathrm{L}^{k_1,\dots,k_\ell}(\mathrm{RT}) & \mathcal{W}_d^{k_1,\dots,k_\ell} \\
			\mathrm{L}^{k^\varphi_1,\dots,k^\varphi_\ell}(\mathrm{RT}) & \mathcal{W}_d^{k^\varphi_1,\dots,k^\varphi_\ell}
			\arrow["F^{k_1,\dots,k_\ell}", from=1-1, to=1-2]
			\arrow["\varphi_*"', from=1-1, to=2-1]
			\arrow["\varphi_*"', from=1-2, to=2-2]
			\arrow["F^{k^\varphi_1,\dots,k^\varphi_\ell}", from=2-1, to=2-2]
		\end{tikzcd}\]
		The commutation of this diagram encodes the trivial fact that identifying the labels in the tree and then taking the elementary differential is exactly the same as first taking the elementary differential and then identifying the functions in the elementary differential.
		The maps $G^{k_1,\dots,k_\ell}$ in dimension $1$ for the multi-indices are defined the same way.
	\end{definition}

	After these definitions of labelling, we are ready to lower the requirement of combinatorial descriptions.
	\begin{definition} \label{l_weak_combinatorial}
		Let $\ell\in\mathbb{N}$ an \emph{$\ell$-weak combinatorial description} of the elementary differentials in dimension $d$ is a sequence of sets $\mathrm{CD}_d=(\mathrm{CD}_d(n))_{n\in\mathbb{N}}$ and a family of maps $H_d^{k_1,\dots,k_\ell}:\mathrm{L}^{k_1,\dots,k_\ell}(\mathrm{CD}_d)\to\mathcal{W}_d^{k_1,\dots,k_\ell}$ such that $H_d^{k_1,\dots,k_\ell}$ is $\mathfrak{S}_n$-equivariant with $n=\sum_{i=1}^{\ell} k_i$, the set $H_d^{k_1,\dots,k_\ell}(\mathrm{L}(\mathrm{CD}_d))$ generates $\mathcal{W}_d^{k_1,\dots,k_\ell}$, and for any identification map $\varphi_*$, the following diagram commutes:
                \[\begin{tikzcd}
                        \mathrm{L}^{k_1,\dots,k_\ell}(\mathrm{CD}_d) & \mathcal{W}_d^{k_1,\dots,k_\ell} \\
                        \mathrm{L}^{k^\varphi_1,\dots,k^\varphi_\ell}(\mathrm{CD}_d) & \mathcal{W}_d^{k^\varphi_1,\dots,k^\varphi_\ell}
			\arrow["H_d^{k_1,\dots,k_\ell}", from=1-1, to=1-2]
                        \arrow["\varphi_*"', from=1-1, to=2-1]
                        \arrow["\varphi_*"', from=1-2, to=2-2]
			\arrow["H_d^{k^\varphi_1,\dots,k^\varphi_\ell}", from=2-1, to=2-2]
                \end{tikzcd}\]
		An $\ell$-weak combinatorial description is \emph{faithful} if for all $n\in\mathbb{N}$, the set $H_d^{k_1,\dots,k_\ell}(\mathrm{L}(\mathrm{CD}_d))$ is a basis of $\mathcal{W}_d^{k_1,\dots,k_\ell}$.
        	A combinatorial description is \emph{compatible with $\mathrm{RT}$} if we have a sequence of equivariant surjective maps $p_n:\mathrm{RT}(n)\to\mathrm{CD}_d(n)$ such that for every $ \tau \in \mathrm{RT}(n) $
        \begin{equs}
		F^{k_1,\dots,k_\ell}((\tau,(x_1,\dots,x_n)) = H^{k_1\dots,k_\ell}_d( (p_n(\tau),(x_1,\dots,x_n) ).
        \end{equs} 
                A combinatorial description is \emph{compatible with $\mathrm{MI}$} if we have a sequence of equivariant surjective maps $q_n:\mathrm{CD}_d(n)\to\mathrm{MI}(n)$ such that for $a \in \mathrm{CD}_d(n) $
                \begin{equs}
			\pi_d(H_d^{k_1,\dots,k_\ell}( (a,(x_1,\dots,x_n) ))) = G^{k_1,\dots,k_\ell}( (q_n(a),(x_1,\dots,x_n)) ),
                \end{equs}
		where $\pi_d$ is the canonical projection $\mathcal{W}_d\to\mathcal{W}_1$.
        \end{definition}	

	We should notice that via the canonical inclusion $[\ell-1]\to[\ell]$, any $\ell$-weak combinatorial description provides an $(\ell-1)$-weak combinatorial description. Moreover, any combinatorial description is in particular an $\ell$-weak combinatorial description.

\section{Weak combinatorial descriptions}

\label{Sec::6}

	The natural question is whether the two main theorems we proved still hold when considering $\ell$-weak combinatorial descriptions.

	\begin{theorem} \label{thm_1_b}
		In dimension $d\neq1$, $\mathrm{RT}$ is the only $2$-weak combinatorial description of $\mathcal{W}_d$ compatible with $\mathrm{RT}$.
	\end{theorem}

	\begin{proof}
		By inspecting the proof of Theorem~\ref{thm:1} we notice that we have only evaluated the elementary differentials with at most $2$ distinct functions. Thus, the same proof holds for $2$-weak combinatorial descriptions.
	\end{proof}

	\begin{theorem} \label{thm_2_b}
		In dimension $d\neq1$, $\mathcal{W}_d$ does not admit any faithful $2d$-weak combinatorial description.
	\end{theorem}

	\begin{proof}
		The obstruction we used to prove Theorem~\ref{thm:2}, is the identity $\mathrm{s}_{2d}$. 
		Thus, for the same proof to hold, it is sufficient to check that this relation is not tautological. 
		We notice that $\mathrm{s}_{2d}$ is not tautological as long as $f_1$ up to $f_{2d}$ are distinct. Thus, the same proof holds as long as we have at least $2d$ distinct functions.
	\end{proof}
	
	\begin{remark}
		Let us explain Theorem~\ref{thm_1_b} in plain words. A $2$-weak combinatorial description of $\mathcal{W}_d$ is the data of two objects, $\mathrm{CD}_d$ and $H_d$. 
		The first one $\mathrm{CD}_d$ is a sequence of combinatorial objects $(\mathrm{CD}_d(n))_{n\in\mathbb{N}})$ with $n$ distinguishable vertices, and an action of $\mathfrak{S}_n$ on $\mathrm{CD}_d(n)$ by permutation of those vertices. 
		The second one $H_d$ associate to any labelling $x=(x_1,\dots,x_n)$ of $c\in\mathrm{CD}_d$ with two colours $1$ and $2$, an elementary differential $H^{k_1,k_2}_d((c,x))\in\mathcal{W}^{k_1,k_2}_d$ where $k_1$ is the number of vertex of colour $1$, and $k_2$ of colour $2$.
Thus, in the definition of a $2$-weak combinatorial description, it is assumed that labelled objects comes from the labelling of combinatorial objects with distinguishable vertices.
		In particular, the Theorem~\ref{thm_1_b} states that for any two trees $\tau_1\neq\tau_2$, we have a $2$-labelling $x=(x_1,\dots,x_n)$ such that
		$$F^{k_1,k_2}((\tau_1,x))\neq F^{k_1,k_2}((\tau_2,x)).$$
		One should be careful since it does not mean that the above inequality hold for all $2$-labelling (which is clearly false, see Theorem~\ref{compatible_trees_b}).
	\end{remark}

	\begin{remark} \label{remark_SPDE}
		Although we just discussed the limitations of Theorem~\ref{thm_1_b}, it already covers some examples for singular SPDEs. 
		Let us consider the following equation:
                \begin{equ}[e:2noises]
			\d_t u = \d_x^2 u + F_1(u,\d_x u)\, \xi_1 + F_2(u,\d_x u)\, \xi_2\;,
                \end{equ}
		with $\xi_1,\xi_2$ two independent space-time noises.
		Here, we should notice that the two non-linearities $F_1$ and $F_2$ play symmetric roles.
		Because $F_1$ and $F_2$ can be permuted, for any decorated tree $\tau$ with a given vertex $v$ labelled by $\xi_1$, the decorated tree $\tau'$ where $v$ is labelled by $\xi_2$, and every else is identical to $\tau$ also needs to be considered. 
		Thus, it is possible to apply Theorem~\ref{thm_1_b} in this context, and we get the following result. 
		\begin{itemize}
			\item The elementary differentials are encoded via the $2$-labellings of combinatorial objects with distinguishable vertices (the fact that the vertices a distinguishable is needed in order to be able to label them). Let us denote by $\mathrm{CD}(n)$ the set of such combinatorial objects with $n$ numbered vertices.
			\item The sets $\mathrm{CD}(n)$ are compatible with the set $\mathrm{RT}(n)$, meaning that $\mathrm{CD}(n)$ is a quotient of $\mathrm{RT}(n)$ via a projection $p_n$, and the elementary differential associated to $c\in\mathrm{CD}(n)$ is the same as the elementary differential associated to $\tau\in\mathrm{RT}(n)$ if $p_n(\tau)=c$.
		\end{itemize}
		Then the only possible choice is to use $\mathrm{RT}$. We get an analogous result in numerical analysis. Let us consider the equation:
                \begin{equ}[e:2splitODE]
			\d_t u = F_1(u) + F_2(u) \;.
                \end{equ}
		Here, the elementary differentials are used to write down the B-series associated with some numerical methods.
		Then the only possible choice is to use $\mathrm{RT}$.

%		Then by Theorem~\ref{thm_1_b}, if one wants to encode elementary differential in a combinatorial way to write B-series expansion, such that:
%		\begin{itemize}
%			\item The elementary differentials are encoded via the $2$-labellings of combinatorial objects with distinguishable vertices (the fact that the vertices are distinguishable is needed to be able to label them). Let us denote by $\mathrm{CD}(n)$ the set of such combinatorial objects with $n$ numbered vertices.
%			\item The sets $\mathrm{CD}(n)$ are compatible with the set $\mathrm{RT}(n)$, meaning that $\mathrm{CD}(n)$ is a quotient of $\mathrm{RT}(n)$ via a projection $p_n$, and the elementary differential associated to $c\in\mathrm{CD}(n)$ is the same as the elementary differential associated to $\tau\in\mathrm{RT}(n)$ if $p_n(\tau)=c$.
%		\end{itemize}
%		Then the only possible choice is to use $\mathrm{RT}$.
%Since this holds when splitting into two, it also holds when splitting into more than two components.
	\end{remark}

	The case of $1$-weak combinatorial descriptions is much more complicated than the $\ell$-weak combinatorial description for $\ell\neq 2$. Indeed, all the labels being the same, one can always quotient by the group action.

	\begin{theorem} \label{compatible_trees_b}
		Let $\mathrm{RT}_\mathfrak{S}$ defined by $\mathrm{RT}_\mathfrak{S}(n)=\mathrm{RT}(n)/\mathfrak{S}_n$. 
		Then $\mathrm{RT}_\mathfrak{S}$ is a $1$-weak combinatorial description of $\mathcal{W}_d$ compatible with $\mathrm{RT}$.
	\end{theorem}

	It would be very interesting to know if one can still get a result similar to Theorem~\ref{thm:1} in the case of $1$-weak combinatorial description. Namely answering the following question:
	\begin{itemize}
		\item[(Q1)] Is $\mathrm{RT}_\mathfrak{S}$ the only $1$-weak combinatorial description of $\mathcal{W}_d$ compatible with $\mathrm{RT}_\mathfrak{S}$ in dimension $d\neq 1$?
	\end{itemize}

	\begin{remark} \label{remark_SPDE}
		Let us stress that answering this question would cover interesting examples for singular SPDEs not covered by Theorem~\ref{thm_1_b}. The stochastic geometric heat equation is given in local coordinates by
		\begin{equ}[e:genClass_intro]
			\d_t u^\alpha = \d_x^2 u^\alpha + \Gamma^\alpha_{\beta\gamma}(u)\,\d_x u^\beta\d_x u^\gamma +  \sigma_i^\alpha(u)\, \xi_i\;,
		\end{equ}
		where $ i \in \lbrace 1,...,m \rbrace $ and the functions
		$\Gamma^\alpha_{\beta\gamma},\sigma_i^\alpha:\mathbb{R}^d\to\mathbb{R}$ with $\Gamma^\alpha_{\beta\gamma}=\Gamma^\alpha_{\gamma\beta}$ are smooth. 
		Here, the $\xi_i$ are independent space-time noises. Such dynamics describe a stochastic evolution taking values in the loops on a Riemannian manifold. 
		In the definition of the elementary differentials associated with the decorated trees arising from the local perturbative expansion of the solution of  \eqref{e:genClass_intro}, one will have derivatives in $u$ performed on the $\Gamma^\alpha_{\beta\gamma}(u)$ and the $\sigma_i^\alpha(u)$. 
		Because $\Gamma^\alpha_{\beta\gamma}$, and $\sigma_i^\alpha$ does not play symmetric roles, thus applying Theorem~\ref{thm_1_b} would be very superficial. 
		Indeed, replacing a $\Gamma^\alpha_{\beta\gamma}$ by a $\sigma_i^\alpha$ in a decorated $\tau$ is a quite unnatural operation. The tree and its elementary differential are still well-defined, but because $\Gamma^\alpha_{\beta\gamma}$ and $\sigma_i^\alpha$ play different roles, it should not be considered in the local expansion of the solution.

		We should point out that in the case of dimension $2$ with a space-time white noise, $54$ decorated trees appear for describing the counter-terms of the renormalised equation with a maximal number of vertices being $7$. The same computer program as in Proposition~\ref{prp:comp} can be used to check that those $54$ trees do not over-parametrise their respective elementary differentials.
	\end{remark}

	For Theorem~\ref{thm:2}, it cannot be adapted to $1$-weak combinatorial description. Indeed, we have the following result:

	\begin{theorem} \label{compatible_multi_b}
		There is a faithful $1$-weak combinatorial description of $\mathcal{W}_d$ compatible with $\mathrm{MI}_\mathfrak{S}$.
	\end{theorem}

	\begin{proof}
		We know that $\mathrm{MI}$ is a faithful combinatorial description of $\mathcal{W}_1$, and thus $\mathrm{MI}_\mathfrak{S}$ is a faithful $1$-weak combinatorial description of $\mathcal{W}_1$. 
		Thus, we have a basis of $\mathcal{W}^n_1$ indexed by $\mathrm{MI}_\mathfrak{S}(n)$ for each $n\in\mathbb{N}$, by a slight abuse of notation, let us identify those bases with $\mathrm{MI}_\mathfrak{S}$. 
		For $m\in\mathrm{MI}_\mathfrak{S}$, let $\mathcal{W}^n_d[m]=\mathrm{Span}(\pi_d^{-1}(m))$ with $\pi_d:\mathcal{W}^n_d\to\mathcal{W}^n_1$ the canonical projection. We have:
		$$\mathcal{W}^n_d=\oplus_{m\in\mathrm{MI}_\mathfrak{S}(n)}\mathcal{W}^n_d[m]$$
		Thus, it is sufficient to choose a basis $B_m$ of $\mathcal{W}^n_d[m]$ for each $m\in\mathrm{MI}_\mathfrak{S}(n)$ such that $\pi_d(B_m)=\{m\}$, which exists by definition of $\mathcal{W}^n_d[m]$.
	\end{proof}
	
	We could point out that although we proved that a faithful $1$-weak combinatorial description compatible with $\mathrm{MI}_\mathfrak{S}$ does exist in any dimension $d$, we did not construct it for $d\neq 1$, and this $1$-weak combinatorial description is very mysterious. For example, not even its arity-wise cardinality is known, and in fact knowing its arity-wise dimension would already be a huge leap forward in the study of polynomial identities of the Witt algebra $W_d$.

\section{Discussion of the results}

	Let us end this article with an informal discussion of the results we obtained. 
	The main goal of this article is to tackle the question of the combinatorial description of elementary differentials.
	Indeed, elementary differentials have proven to be a useful and powerful tool both in numerical analysis and SPDEs.
	The rooted trees are a convenient combinatorial description of those elementary differentials, allowing to use combinatorics, and algebra to study B-series, and even leading to a categorical approach of B-series.
	However, it is well known that for any fixed dimension $d$, the rooted trees over-parametrise the elementary differentials. 
	In the case of dimension $1$, there is a way to bypass this over-parametrisation using multi-indices. 
	The ultimate goal would be to find similar combinatorial objects in dimension $d\neq 1$ with sufficiently nice properties so that it could actually be used.
	Let us give three kinds of properties it should satisfy:
	\begin{itemize}
		\item It should respect the algebraic structure of $\mathcal{W}_d$.
		\item The over-parametrisation should be kept at a minimum.
		\item It should be compatible with the usual interpretation with rooted trees and multi-indices.
	\end{itemize}
	Let us start with the compatibility with the algebraic structure of $\mathcal{W}_d$.
	The space $\mathcal{W}_d$ has a very rich algebraic structure (it is in fact an algebraic operad). In this article, we gave two notions of compatibility. The \emph{combinatorial descriptions} and the \emph{$\ell$-weak combinatorial descriptions}. One should point out that the $1$-weak combinatorial descriptions behave in a very different way than the $\ell$-weak combinatorial descriptions. 
	\begin{itemize}
		\item The combinatorial descriptions are the ones respecting the most the algebraic structure of $\mathcal{W}_d$. With such a combinatorial description, the splitting method would generalise in a straightforward way, and it is possible to tell in which ``vertex'' we put each function when evaluating the corresponding elementary differential.
		\item The $\ell$-weak combinatorial descriptions with $\ell\neq 1$ respect less algebraic structure. The splitting method would still generalise, but only up to $\ell$ distinct functions. However, it is still possible to distinguish the ``vertices''.
		\item In the case $\ell=1$, all the group actions become trivial, and no algebraic structure is preserved. ``Vertices'' are no longer distinguishable, and we are left to find a basis of a vector space.
	\end{itemize}
	For the over-parametrisation, we define the notion of \emph{faithful} description. 
	\begin{itemize}
		\item In the case of combinatorial descriptions, and $\ell$-weak combinatorial descriptions with $\ell\neq 1$, then being faithful means that no linear relation is satisfied aside from the tautological one. The tautological relations are the relations arising from the fact that if two ``vertices'' are evaluated with the same function, then the permutation of those vertices does not change the result. In this case, both the rooted trees and the multi-indices are faithful.
		\item When $\ell=1$, the situation is quite different. Indeed, a lot of tautological relations appear, and the rooted trees and the multi-indices are no longer faithful. One needs to take the quotient by the group action.
	\end{itemize}
	We defined two notions of compatibility:
	\begin{itemize}
		\item the compatibility with the rooted trees, which is the strongest one,
		\item and the compatibility with the multi-indices, which is weaker.
	\end{itemize}
	The compatibility with the rooted trees is indeed stronger, since one may check that it implies compatibility with the multi-indices.
	In the case of $1$-weak combinatorial descriptions, the natural compatibility conditions are compatibility with $\mathrm{RT}_\mathfrak{S}$, and with $\mathrm{MI}_\mathfrak{S}$ since all the group actions become trivial.

	One can ask if is there any ``meaningful'' combinatorial way to describe elementary differentials in fixed dimension, that is not the rooted trees?

	We would like to argue that there is not. The first reason is Theorem~\ref{thm_1_b} that states that if we impose any compatibility with the algebraic structure of $\mathcal{W}_d$, and compatibility with $\mathrm{RT}$, the only possibility is $\mathrm{RT}$. The second reason is Theorem~\ref{thm:3} stating that for $d=2$ there is no faithful combinatorial description, and that we conjecture to hold for $d\neq 1$. Moreover, it sounds reasonable not to expect any faithful $2$-weak combinatorial descriptions.

	Thus, the only options are either a $1$-weak combinatorial description compatible with $\mathrm{RT}_\mathfrak{S}$, and less over-parametrised than $\mathrm{RT}_\mathfrak{S}$ that may exist, or a combinatorial description not compatible with $\mathrm{RT}$, and not faithful. None of those possibilities seems to be a ``satisfactory'' way to describe elementary differentials in a fixed dimension.


\begin{thebibliography}{99}
















   \bibitem{BB19}
  I. Bailleul, I. Bernicot. \newblock { \em High order paracontrolled calculus}. Forum of Mathematics, Sigma, 7, E44. 
  \newblock \burlalt{doi:10.1017/fms.2019.44}{https://doi.org/10.1017/fms.2019.44}.
  
  

\bibitem{BB24}
C.~Bellingeri, Y.~Bruned.
\newblock {\em Symmetries for the gKPZ equation via multi-indices
}
\newblock \burlalt{arXiv:2410.00834 }{http://arxiv.org/abs/2410.00834}.

\bibitem{BBH25}
C.~Bellingeri, Y.~Bruned, Y.~Hou.
\newblock {\em Flows driven by multi-indices Rough Paths
}
\newblock \burlalt{arXiv:2502.14771 }{http://arxiv.org/abs/2502.14771}.



\bibitem{BLL}
	{\rm F. Bergeron, G. Labelle, P. Leroux}. 
	{\em Combinatorial species and tree-like structures.} Encyclopedia of Mathematics and its Applications. \textbf{Vol. 67}. Cambridge University Press, Cambridge, (1998), xx+457.
	\burlalt{doi:10.1017/CBO9781107325913}{https://doi.org/10.1017/CBO9781107325913}




\bibitem{BCCH}
 { \rm Y. Bruned, A. Chandra, I. Chevyrev,
  M. Hairer}.
\newblock {\em Renormalising SPDEs in regularity structures}.
\newblock J. Eur. Math. Soc. (JEMS), \textbf{23}, no.~3, (2021), 869-947.
\newblock
  \burlalt{doi:10.4171/JEMS/1025}{http://dx.doi.org/10.4171/JEMS/1025}.
  
  
  \bibitem{BD23}
  Y.~Bruned, V. Dotsenko.
  \newblock {\textsl{Novikov algebras and multi-indices in regularity structures.}}
  \newblock \burlalt{arXiv:2311.09091  }{http://arxiv.org/abs/2311.09091 }.
  
  \bibitem{BD24}
  Y.~{Bruned}, V.~{Dotsenko}. {\em Chain rule symmetry for singular SPDEs}. 
  \newblock \burlalt{arXiv:2403.17066}{http://arxiv.org/abs/2403.17066}.
  
  \bibitem{BHE24}
  Y.~Bruned, K.~Ebrahimi-Fard, Y.~Hou. \newblock { \em Multi-indice B-series}. \newblock J. Lond. Math. Soc. \textbf{111}, no.~1, (2025), e70049.
  \newblock
  \burlalt{doi:10.1112/jlms.70049}{https://doi.org/10.1112/jlms.70049}.



\bibitem{BGHZ}
Y.~Bruned, F.~Gabriel, M.~Hairer, L.~Zambotti.
 {\em Geometric stochastic heat equations}
\newblock J. Amer. Math. Soc. (JAMS), \textbf{35}, no.~1, (2022), 1-80.
\newblock
  \burlalt{doi.org/10.1090/jams/977}{http://dx.doi.org/10.1090/jams/977}.
 

\bibitem{BHZ}
{\rm Y. Bruned, M. Hairer, L. Zambotti}.
\newblock {\em Algebraic renormalisation of regularity structures.}
\newblock Invent. Math. \textbf{215}, no.~3, (2019), 1039--1156.
\newblock
  \burlalt{doi:10.1007/s00222-018-0841-x}{https://dx.doi.org/10.1007/s00222-018-0841-x}.
  
 
  
  
 
  
  
%\bibitem{EMS}
%  {\rm Y. Bruned, M. Hairer, L. Zambotti}.
%\newblock {\em Renormalisation of Stochastic Partial Differential Equations.}
%\newblock EMS Newsletter \textbf{115}, no.~3, (2020), 7--11.
%\newblock
%  \burlalt{doi: 10.4171/NEWS/115/3}{http://dx.doi.org/10.4171/NEWS/115/3}.



  
 
 \bibitem{BL24}
Y.~{Bruned}, P.~{Linares}.
\newblock {\em  A top-down approach to algebraic renormalization in regularity structures based on multi-indices.} Arch. Ration. Mech. Anal. \textbf{248}, 111 (2024). 
\burlalt{doi:10.1007/s00205-024-02041-4}{http://dx.doi.org/10.1007/s00205-024-02041-4}. 


\bibitem{BM25}
Y.~Bruned, A.~Minguella.
\newblock {\em Renormalisation in the flow approach for singular SPDEs
}
\newblock \burlalt{arXiv:2504.04885 }{http://arxiv.org/abs/2504.04885}.



\bibitem{BS}
Y.~{Bruned}, K.~{Schratz}. \newblock { \em Resonance based schemes for dispersive equations via decorated 
	trees}. Forum of Mathematics, Pi, 10, E2. 
\newblock \burlalt{doi:10.1017/fmp.2021.13}{https://doi.org/10.1017/fmp.2021.13}.


\bibitem{Butcher72}
{ \rm J. C. Butcher}.
\newblock {\em An algebraic theory of integration methods.}
\newblock Math. Comp. \textbf{26}, (1972), 79--106.
\newblock \burlalt{doi:10.2307/2004720}{http://dx.doi.org/10.2307/2004720}.
















 
 \bibitem{CF24}
 A.~{Chandra}, L.~{Ferdinand}. {\em A flow approach to the generalized KPZ equation}. 
 \newblock \burlalt{arXiv:2402.03101}{http://arxiv.org/abs/2402.03101}.	
	

  \bibitem{CL01}
  F.~Chapoton, M.~Livernet,
  {\textsl{Pre-Lie algebras and the rooted trees operad}},
  Internat.~Math.~Res.~Notices  \textbf{2001}, no.~8, (2001), 395--408.
  \burlalt{doi:10.1155/S1073792801000198}{https://doi.org/10.1155/S1073792801000198}. 


\bibitem{CM07}
	{\rm P. Chartier, A. Murua},
	{\em Preserving first integrals and volume forms of additively split systems.}
	IMA Journal of Numerical Analysis, \textbf{27}, Issue 2, (2007), 381--405. 
	\burlalt{doi:10.1093/imanum/drl039}{https://doi.org/10.1093/imanum/drl039}



 \bibitem{DNY22}
{\rm  Y. Deng · A. R. Nahmod,
	H. Yu}.
\newblock {\em Random tensors, propagation of randomness, and nonlinear dispersive equations.}
\newblock Invent. Math. \textbf{228}, no.~3, (2022), 539--686.
\newblock
\burlalt{doi:10.1007/s00222-021-01084-8}{https://dx.doi.org/10.1007/s00222-021-01084-8}.




\bibitem{DH23}
{\rm Y.~Deng,  Z.~Hani}.
\newblock \emph{Full derivation of the wave kinetic equation}.
\newblock  Invent. math. \textbf{233}, (2023),  543-724.
\newblock
\burlalt{doi:10.1007/s00222-023-01189-2}{https://dx.doi.org/10.1007/s00222-023-01189-2}.

\bibitem{DL25}
	{\rm V. Dotsenko, P. Laubie}.
	{\em Volume preservation of Butcher series methods from the operad viewpoint.}
	Internat.~Math.~Res.~Notices \textbf{2025}, no.~13, (2025).
	\burlalt{doi:10.1093/imrn/rnaf187}{https://doi.org/10.1093/imrn/rnaf187}.

\bibitem{Duc21}
P.~Duch.
\newblock{\textsl{Flow equation approach to singular stochastic PDEs}}  Probab. and Math. Phys. \textbf{6}, no.~2, (2025), 327–437. T
\newblock
\burlalt{doi:10.2140/pmp.2025.6.327}{https://dx.doi.org/10.2140/pmp.2025.6.327}.


\bibitem{D00}
A.  Dzhumadil'daev.
\newblock {\em Minimal identities for right-symmetric algebras.} J. Algebra
\textbf{225}, no. 1, (2000), 201--230
\burlalt{doi:10.1006/jabr.1999.8109}{https://doi.org/10.1006/jabr.1999.8109}.

\bibitem{DL}
{ \rm  A.~Dzhumadil'daev, C.~Löfwall}.
\newblock {\em  Trees, free right-symmetric algebras, free Novikov algebras and identities.}
\newblock Homology Homotopy Appl. \textbf{4}, no.~2, (2002), 165--190.






\bibitem{E94}
	{\rm Y. B. Ermolaev}. 
	{\em On a certain algebra on a set of graphs.}
	Sib Math J. \textbf{35}, (1994), 706–712. 
	\burlalt{doi:10.1007/BF02106613}{https://doi.org/10.1007/BF02106613}


	
	










\bibitem{GIP15}
M.~Gubinelli, P.~Imkeller, N.~Perkowski.
\newblock {\textsl{Paracontrolled distributions and singular PDEs}}
\newblock Forum of mathematics, Pi, \textbf{3}, (2015), 3:e6.
\newblock 
  \burlalt{doi:10.1017/fmp.2015.2}{http://dx.doi.org/10.1017/fmp.2015.2}.
  
  \bibitem{Gubinelli2004}
  {\rm M.~Gubinelli}.
  \newblock \emph{Controlling rough paths}.
  \newblock J. Funct. Anal. \textbf{216}, no.~1, (2004), 86
  -- 140.
  \newblock
  \burlalt{doi:10.1016/j.jfa.2004.01.002}{https://dx.doi.org/10.1016/j.jfa.2004.01.002}.
  

\bibitem{Gub06}
{\rm M.~Gubinelli}.
\newblock \emph{Ramification of rough paths}.
\newblock J. Differ. Equ. \textbf{248}, no.~4, (2010), 693 -- 721.
\burlalt{doi:10.1016/j.jde.2009.11.015}{https://dx.doi.org/10.1016/j.jde.2009.11.015}.


\bibitem{reg}
{\rm M. Hairer}.
\newblock {\em A theory of regularity structures.}
\newblock Invent. Math. \textbf{198}, no.~2, (2014), 269--504.
\newblock
  \burlalt{doi:10.1007/s00222-014-0505-4}{https://dx.doi.org/10.1007/s00222-014-0505-4}. 
  
  



	\bibitem{HK15}
{\rm M.~Hairer, D.~Kelly}.
\newblock \emph{Geometric versus non-geometric rough paths}.
\newblock Ann. Inst. H. Poincaré Probab. Statist. \textbf{51}, no.~1,
(2015), 207--251.
\newblock
\burlalt{doi:10.1214/13-AIHP564}{https://dx.doi.org/10.1214/13-AIHP564}.

\bibitem{HNW}
	{\rm E. Hairer, S. Nørsett, G. Wanner},
	{\em Solving ordinary differential equations. I}
	 Vol. 8. Springer Series in Computational Mathematics. Nonstiff problems. Springer-Verlag, Berlin, 1987, pp. xiv+480. \burlalt{doi:10.1007/978-3-540-78862-1}{https://doi.org/10.1007/978-3-540-78862-1}.

\bibitem{HW74}
{\rm E. Hairer, G. Wanner},
{\em On the Butcher group and general multi-value methods.} Computing, \textbf{13}, (1974), 1--15.  \burlalt{doi:10.1007/BF02268387}{https://dx.doi.org/10.1007/BF02268387}.


\bibitem{IQT}
	{\rm A. Iserles, G. Quispel, P. Tse},
	{\em B-series methods cannot be volume-preserving.} 
	Bit Numer Math \textbf{47}, (2007), 351--378. 
	\burlalt{doi:10.1007/s10543-006-0114-8}{https://doi.org/10.1007/s10543-006-0114-8}



\bibitem{J81}
	{\rm A. Joyal}, 
	{\em Une théorie combinatoire des séries formelles.}
	Adv. in Math. \textbf{42.1}, (1981), 1--82.
	\burlalt{doi:10.1016/0001-8708(81)90052-9}{https://doi.org/10.1016/0001-8708(81)90052-9}



 
\bibitem{KU16}
	{\rm D. Kozybaev, U. Umirbaev}, 
	{\em Identities of the left-symmetric Witt algebras.}
	Int. J. Algebra Comput. \textbf{26}, No. 2, (2016), 435--450. \burlalt{doi:10.1142/S021819671650017X}{https://doi.org/10.1142/S021819671650017X}.






	
 \bibitem{Li23}
P.~Linares.
\newblock {\em Insertion pre-Lie products and translation of rough paths based on multi-indices}. To appear in Ann. Inst. Henri Poincar\'{e} Probab. Stat. 
\newblock \burlalt{arXiv:2307.06769
}{https://arxiv.org/abs/2307.06769}.



\bibitem{LOT}
P.~{Linares}, F.~{Otto}, M.~{Tempelmayr}. \newblock { \em The structure group for quasi-linear equations via universal enveloping algebras}. 
\newblock {Comm. Amer. Math. Soc. \textbf{3}, (2023), 1-64.}
\newblock \burlalt{doi:10.1090/cams/16}{https://doi.org/10.1090/cams/16}.

\bibitem{LV}
	{\rm J.-L. Loday, B. Vallette}, 
	{\em Algebraic operads.} Grundlehren der mathematischen Wissenschaften [Fundamental Principles of Mathematical Sciences]. \textbf{Vol. 346}. Springer, Heidelberg, (2012), xxiv+634. 
	\burlalt{doi:10.1007/978-3-642-30362-3}{https://doi.org/10.1007/978-3-642-30362-3}

\bibitem{lyons1998}
T. Lyons.
\newblock {\em Differential equations driven by rough signals.}
Rev. Mat. Iberoam. \textbf{14}, no. 2, (1998).215--310
\burlalt{doi:10.4171/rmi/240}{https://doi.org/10.4171/rmi/240}.

\bibitem{MV16}
H. Munthe-Kaas, O.~Verdier.
\newblock \emph{Aromatic Butcher Series}.
\newblock Found. Comput. Math. \textbf{16}, (2016), 183--215.
\newblock
\burlalt{doi:10.1007/s10208-015-9245-0}{http://dx.doi.org/10.1007/s10208-015-9245-0}.






\bibitem{OSSW}
F.~Otto, J.~Sauer, S.~Smith, H.~Weber.
\newblock {\em A priori bounds for quasi-linear SPDEs in the full sub-critical regime}. J. Eur. Math. Soc. (JEMS)  \textbf{27}, no. 1, (2025),  71--118.
\burlalt{doi:10.4171/JEMS/1574}{http://dx.doi.org/10.4171/JEMS/1574}. 









\end{thebibliography}
\end{document}